\documentclass[a4paper,11pt,english]{smfart}
\usepackage{amsfonts}
\usepackage{amsmath,amsthm}

\usepackage{latexsym}
\usepackage{array}
\usepackage{amssymb}
\usepackage{graphicx}
\usepackage{enumerate}
\usepackage{verbatim}
\usepackage{float}
\usepackage[english]{babel}

\usepackage{color}
\definecolor{marin}{rgb}   {0.,   0.3,   0.7} 
\definecolor{rouge}{rgb}   {0.8,   0.,   0.} 
\definecolor{sepia}{rgb}   {0.8,   0.5,   0.} 
\usepackage[colorlinks,citecolor=marin,linkcolor=rouge,
            bookmarksopen,
            bookmarksnumbered
           ]{hyperref}

\theoremstyle{plain} 
\newtheorem{theorem}{Theorem}[section]
\newtheorem{lemma}[theorem]{Lemma}
\newtheorem{proposition}[theorem]{Proposition}
\newtheorem{corollary}[theorem]{Corollary} 
 \theoremstyle{remark}

\newtheorem{example}[theorem]{Example}
 
\newtheorem{ass}[theorem]{Assumption}

\newcommand{\R}{  \mathbb{R}   }

\newcommand{\Z}{  \mathbb{Z}   }
\newcommand{\N}{  \mathbb{N}   }

\newcommand{\wb}{\boldsymbol{w}}
\newcommand{\Fb}{\boldsymbol{F}}
\newcommand{\Wb}{\boldsymbol{\mathsf{W}}}
\newcommand{\Us}{\mathsf{U}}
\newcommand{\Vs}{\mathsf{V}}
\newcommand{\Zs}{\mathsf{Z}}
\newcommand{\Rb}{\boldsymbol{R}}
\newcommand{\Sb}{\boldsymbol{S}}
\newcommand{\Gb}{\boldsymbol{\mathsf{G}}}
\newcommand{\gb}{\boldsymbol{g}}
\newcommand{\rb}{\boldsymbol{r}}
\newcommand{\mb}{\boldsymbol{m}}
\newcommand{\psib}{\boldsymbol{\psi}}

\newcommand{\Cb}{\mathbb{C}}

\newcommand{\dd}{\mathrm{d}}

\newcommand{\Hc}{ \mathcal{H}   }
\renewcommand{\O}{  \mathcal{O}   }

\newcommand{\T}{  \mathbb{T}   }

\newcommand{\Ec}{  \mathcal{E}   }

\renewcommand{\phi}{  \varphi  }

\newcommand{\be}{\begin{equation}}
\newcommand{\ee}{\end{equation}}
\newcommand{\ben}{\begin{equation*}}
\newcommand{\een}{\end{equation*}}

\newcommand{\nab}{\langle \nabla \rangle_c}
\newcommand{\Norm}[2]{\|#1\|\left.\vphantom{T_{j_0}^0}\!\!\right._{#2}}

\numberwithin{equation}{section}

\author{Erwan Faou}
\address{INRIA \& ENS Cachan Bretagne  \\
Avenue Robert Schumann F-35170 Bruz, France. } 
\email{Erwan.Faou@inria.fr}

 \author{Katharina Schratz}
\address{INRIA \& ENS Cachan Bretagne  \\
Avenue Robert Schumann F-35170 Bruz, France. } 
\email{Katharina.Schratz@inria.fr}

\begin{document}
\title[AP schemes for the KG equation in the non-relativistic regime]
{Asymptotic preserving schemes for the  Klein-Gordon equation in the non-relativistic limit regime}

\maketitle
\begin{abstract}
We consider the Klein-Gordon equation in the non-relativistic limit regime, i.e. the speed of light $c$ tending to infinity. We construct an asymptotic expansion for the solution with respect to the small parameter depending on the inverse of the square of the speed of light. As the first terms of this asymptotic can easily be simulated our approach allows us to construct numerical algorithms that are robust with respect to the large parameter $c$ producing high oscillations in the exact solution. 
\end{abstract}

\subjclass{  }
\thanks{
}

%
%
%
%
%


\section{Introduction}
The Klein-Gordon equation is a fundamental physical equation describing the motion of a spin-less particle, see for instance \cite{BjoD65}, \cite{Sak67} and \cite{Bog59} for a physical introduction on this topic and \cite{GinVe85}, \cite{BreWa81} and the references therein for the mathematical analysis, and in particular \cite{Bou96,Bam03} for some long time results when the equation is set on a compact domain, the case mainly considered in the present work.

 In the relativistic regime, i.e. the speed of light $c=1$, the Klein-Gordon equation (as a non-linear wave equation) is numerically well studied, see for instance \cite{StVaz78}, \cite{JimVaz78}, \cite{PasJimVaz95}, and,  for long time behavior results in the weakly nonlinear regime, \cite{CoHaLu08}, \cite{FG11} and \cite{FGP1,FGP2}. Here we are interested in the Klein-Gordon equation in the non-relativistic limit regime, i.e. the speed of light $c$ tending to infinity. Thus, our model problem reads
\begin{equation}
\label{eq:kgr}
\frac{1}{c^2} \partial_{tt} z - \Delta z + c^2 z = f(z), \quad z(0) = \phi, \quad  \partial_t z(0) = c^2 
\gamma,\quad c\gg 1,
\end{equation}
where we will consider essentially the case $f(z) =\lambda  |z|^{2p} z$, $\lambda \in \R$, assume that the initial condition $\varphi$ and $\gamma$ do not depend on $c$, and consider periodic boundary conditions in the space variable.  In this regime the Klein-Gordon equation has recently gained a lot of attention, see for instance \cite{MaNakOz02}, \cite{MaNak02} and \cite{Tsu84} (when Equation \eqref{eq:kgr} is set on the whole space $\R^d$). Numerically, the highly oscillatory nature of the exact solution is very challenging as standard numerical schemes require severe time step restrictions depending on the large parameter $c^2$, see for instance \cite{HLW}, \cite{EFHI09} and \cite{PetzJ97} for a numerical overview on highly oscillatory problems. In the recent work \cite{BaoD12}, Gautschi-type exponential integrators (proposed in \cite{HoLu99} for equations with high frequencies generated by the linear part) were applied to the non-relativistic Klein-Gordon equation, which allow time steps of order $\mathcal{O}(c^{-2})$. 

Following the one-term asymptotics results given in \cite{MaNak02} we  study in this paper the existence of a complete asymptotic expansion of the exact solution of \eqref{eq:kgr} in terms of $c^{-2}$ over a $c$ -independent interval $(0,T)$, and for smooth initial data. This will allow us to construct numerical schemes which do not need to obey any $c$-dependent smallness restriction. In a nutshell, the asymptotic expansion up to the first-order correction term of \eqref{eq:kgr} reads
\begin{equation}\label{eq:asympz}
z(t,x) = \textstyle \frac12 u_0(t,x) \mathrm{e}^{ic^2t} + \textstyle \frac12 \overline{v_0(t,x)} \mathrm{e}^{-ic^2t} + \mathcal{O}(c^{-2}),\qquad t\in [0,T]
\end{equation}
where $u_0$ and $v_0$ solve the coupled nonlinear Schr\"odinger equations
\begin{equation}\label{eq:nsl0}
\begin{aligned}
&i \partial_t u_0 -\frac{1}{2}\Delta u_0 = \frac{1}{2\pi}\int_0^{2\pi} f(\textstyle \frac{1}{2}(u_0+\mathrm{e}^{-2i\theta}\overline{v_0}))\,\mathrm{d}\theta,\qquad u_0(0) = \varphi - i \gamma, \\
&i \partial_t v_0 -\frac{1}{2}\Delta v_0 = \frac{1}{2\pi}\int_0^{2\pi} f(\textstyle\frac{1}{2}(v_0+\mathrm{e}^{-2i\theta}\overline{u_0}))\,\mathrm{d}\theta,\qquad v_0(0) = \overline{\varphi + i \gamma}.
\end{aligned}
\end{equation}
The essential point is that the highly-oscillating part in the asymptotic expansion \eqref{eq:asympz} is only contained in the phases $\mathrm{e}^{ic^2t}$ and $\mathrm{e}^{-ic^2t}$, but does not appear in the above Schr\"odinger equations for $u_0$ and $v_0$. In particular the numerical time integration of \eqref{eq:nsl0} can be carried out without any $c$-dependent time step restriction. 
Thus, with this asymptotic expansion we are able to obtain a well suited numerical approximation to the exact solution $z(t)$ of the Klein-Gordon equation in the non-relativistic limit $c\gg 1$, without any time step restriction as we only need to solve the Schr\"odinger equations \eqref{eq:nsl0} and multiply the numerical approximations of $u_0$ and $v_0$ with the highly-oscillatory phases $\mathrm{e}^{ic^2t}$ and $\mathrm{e}^{-ic^2t}$, respectively.  

Note that in comparison with \cite{MaNak02} and motivated by numerical implementation, we choose to work here with periodic boundary conditions, that is the space variable $x \in \T^d$ the $d$-dimensional torus. 
Moreover, we show that for sufficiently smooth initial values the asymptotic expansion of the exact solution to \eqref{eq:kgr} exists up to an arbitrary order, and can be reconstructed by superposing highly oscillatory terms to cascaded solutions of $c$-independent Schr\"odinger-like systems. 
 After a full discretization of these cascaded terms, we are able to build numerical schemes which approach the exact solution up to errors terms of order $\mathcal{O}(\tau^k+c^{-2N-2} + h^{s})$ by simulating the terms in the asymptotic expansion up to order $N$ with a numerical scheme of order $k$ in time (with step size $\tau$) and Fourier pseudo-spectral discretization on a grid with mesh size $h$. Note that here, $h$ and $\tau$ can be chosen independently of $c$, and that  $s = s(N)$ depends on the regularity of the solution and of the degree of the asymptotic expansion. 
 
The paper is organized as follows. We commence in Section \ref{sec:boundi} with the essential a priori bounds on the exact solution of the Klein-Gordon equation \eqref{eq:kgr} showing the existence on a time interval independent of the large parameter $c$. In Section \ref{sec:asymp} we will derive an asymptotic expansion of the exact solution in powers of $c^{-2}$. In Section \ref{sec:firstt} we will precisely state the first terms in this expansion, which will allow us to construct numerical schemes which are robust with respect to $c$, see Section \ref{sec:nume}. Furthermore, in Section \ref{sec:enq} we analysis the asymptotic behavior of the conserved physical quantities.

\section{A priori bounds on the exact solution}\label{sec:boundi}

We consider our model problem \eqref{eq:kgr} set on the $d$-dimensional torus $\T^d$. We make the following assumptions: 
\begin{ass}
The initial values $\varphi$ and $\gamma$ are smooth functions independent of c. Furthermore we assume that the nonlinearity $f$ 
is gauge invariant and polynomial of degree $r>1$ in $z$ and $\overline{z}$. This means that for all number $\alpha \in \R$ and all $z \in \Cb$, we have $f(e^{i\alpha}z) = e^{i\alpha} f(z)$. Eventually, we assume that $f$ 
is {\em real} in the sense that for all $z \in \Cb$, we have $\overline{f(z)} = f(\overline{z})$. In particular, this shows that if $z(0)$ and $\partial_t z(0)$ are in $\R$, then the solution $z(t)$ is real for all times.  Typical examples of such nonlinearity are given by the polynomials $f(z) = \lambda |z|^{2p} z$, $p \geq 1$. 
\end{ass}

For ease of presentation and numerical motivations, we will mainly work on the $d$-dimensional torus $\T^d$, $d \in \N \backslash\{0\}$. However, as the only arguments used in the sequel are the local Lipschitz nature of the nonlinearity in Sobolev space and some explicit calculations for diagonal operators  in Fourier, the reader should be convinced that all the result of this paper can be easily extended to functions defined on $\R^d$.
 
For a function $u(x)$, $x \in \T^d$ in $H^s := H^s(\R,\Cb)$ we set 
$$
\Norm{u}{H^s}^2 = \sum_{k \in \Z^d} (1 + |k|)^{2s} |\hat u_k|^2,
$$
where 
$$
\hat u_k = \frac{1}{(2\pi)^d} \int_{\T^d} e^{-i k \cdot x} u(x) \dd x,
$$
$k \in \Z^d$, denote the Fourier coefficients associated with $u$. Here, we have set 
$$
k\cdot x = k_1 x_1 + \cdots + k_d x_d\quad \mbox{and}\quad 
|k|^2 = k_1^2 + \cdots+ k_d^2, 
$$
for $k = (k_1,\ldots,k_d) \in \Z^d$ and $x = (x_1,\ldots,x_d) \in \T^d$. 
For a given $c > 0$, we define the operator
$$
\nab = \sqrt{- \Delta + c^2}. 
$$
In Fourier, the operator $\nab$ is represented by the diagonal operator with coefficients 
$$
(\nab)_{k\ell} = \delta_{k\ell} \sqrt{|k|^2 + c^2}, \quad k,\ell \in \Z^d, 
$$
where $\delta_{k\ell}$ denotes the Kronecker operator. In particular note that the operator $c \nab^{-1}$ is uniformly bounded with respect to $c$: we have for all $k \in \Z^d$, 
$|(c \nab^{-1})_{kk}| \leq 1$, 
which implies that for all functions $u$, $\Norm{c \nab^{-1}u}{H^s} \leq \Norm{u}{H^s}$. 
With these notations, Equation \eqref{eq:kgr} can be written as
\begin{equation}\label{eq:kgrr}
\partial_{tt} z + c^2 \nab^2 z = c^2 f(z). 
\end{equation}
In order to rewrite the above equation as a first-order system in time, we set
\begin{equation}\label{eq:uv}
u = z - i c^{-1}\nab^{-1} \partial_t z , \quad v = \overline z - i c^{-1}\nab^{-1} \partial_t \overline z. 
\end{equation}
Note that we have 
\begin{equation}\label{eq:zuv}
z = \frac12 (u + \overline{v}), 
\end{equation}
and if $z$ is real, then $u = v$. 
A short calculation shows that in terms of the variables $u$ and $v$ equation \eqref{eq:kgrr} reads
\begin{equation}
\label{eq:NLSc}
\begin{array}{rcl}
i \partial_t u &=& -c\nab u + c\nab^{-1} f(\textstyle \frac12 (u + \overline v)), \\[2ex]
i \partial_t v &=& -c\nab v  + c\nab^{-1} f(\textstyle \frac12 (\overline u +  v)).
\end{array}
\end{equation}
The initial boundary conditions associated with this problem are (see \eqref{eq:kgr})
\begin{equation}
\label{eq:BCc}
u(0,x) = \varphi(x) - i c \nab^{-1}\gamma(x), \quad \mbox{and}\quad v(0,x) = \overline{\varphi(x)} - i c \nab^{-1} \overline{ \gamma(x)}. 
\end{equation}
Eventually, we set 
\begin{equation}
\label{eq:rlb}
\wb = \begin{pmatrix} u \\[1ex] v \end{pmatrix}, \quad \boldsymbol{F}(\wb) = \begin{pmatrix} 
f(\textstyle \frac12 (u + \overline v)) \\[1ex] f(\textstyle \frac12 (\overline u +  v))
\end{pmatrix}, \quad \mbox{and}\quad \psib = \begin{pmatrix}
\varphi - i c \nab^{-1} \gamma \\[1ex]
\overline{\varphi} - i c \nab^{-1} \overline{\gamma}
\end{pmatrix}, 
\end{equation}
so that the problem \eqref{eq:NLSc}-\eqref{eq:BCc} can be written as
\begin{equation}
\label{eq:NLF}
i \partial_t \wb = - c \nab \wb + c \nab^{-1} \Fb(\wb), \quad \wb(0) = \psib. 
\end{equation}
We see that this equation is posed on $ H^s \times H^s$.
For an element $\wb = (u,v)^T \in H^s\times H^s$, we set 
$$
\Norm{\wb}{H^s}^2 := \Norm{u}{H^s}^2 + \Norm{v}{H^s}^2. 
$$
In the previous and in the following, if an operator  - like $\nab$ - acts on $H^s$, we naturally extend diagonally its action on $H^s \times H^s$ by setting for instance for $\wb = (u,v)^T$, 
$$
\nab \wb = \nab \begin{pmatrix} u \\ v \end{pmatrix} := \begin{pmatrix} \nab u \\ \nab v \end{pmatrix}. 
$$
The solution of the linear equation ($F = 0$ in the previous equation) can be written as follows: For $\wb = (u,v) \in H^{s} \times H^s$, 
$$
\wb(t) = e^{i t c \nab} \psib, 
$$
where the right hand-side is easily defined in Fourier by the formula
$$
\forall\, k \in \Z^d, \quad 
\widehat{\big(e^{i t c \nab} \psib\big)}_k = 
e^{i t c \sqrt{|k|^2 + c^2}} \, \widehat \psib_k, 
$$
where $\widehat \psib_k$ denote the (vector valued) Fourier coefficients of $\psib$. 
Note that the linear flow is an isometry in the sense that for $\psib \in H^s\times H^s$, 
$$
\Norm{e^{it c \nab} \psib}{H^s}  = \Norm{\psib}{H^s}. 
$$
Note that in the sequel, and because the equation is set on the torus, we will not use any argument involving Strichartz estimates as in \cite{MaNak02} but we will require more regularity on the initial conditions $\gamma$ and $\varphi$.

As we assume that the function $f$ is polynomial of degree $r > 2$ in $z$ and $\bar z$ the application $\Fb$ satisfies estimates of the form 
\begin{equation}
\label{eq:lips}
\Norm{\Fb(\wb_1) - \Fb(\wb_2)}{H^s} \leq C( 1 + \Norm{\wb_1}{H^s} + \Norm{\wb_2}{H^s})^{r-1} \Norm{\wb_1 - \wb_2}{H^s}
\end{equation}
for $s > d/2$, 
with some constants $C$, and for $\wb_1$ and $\wb_2$ in $H^s \times H^s$. 
These estimates are consequences of classical bilinear estimates in $H^s$: 
\begin{equation}
\label{eq:bil1}
\Norm{uv}{H^s} \leq C_s \Norm{u}{H^s}\Norm{v}{H^s}\quad\mbox{for}\quad s > d/2, 
\end{equation}
for some constants $C_s$. 
Moreover the assumption that $f$ is real implies that for all $w \in H^s\times H^s$, $\overline{\Fb(\wb)} = \Fb(\overline{\wb})$. 


With these notations, the mild formulation of \eqref{eq:NLF} is given by 
$$
\wb(t) = e^{i t c \nab} \psib + \int_{0}^t  c \nab^{-1} e^{i (t-s) c \nab} \Fb(\wb(s)) \dd s. 
$$
Using the fact that the linear flow $e^{i t c \nab}$ is an isometry in $H^s$ and that $c \nab^{-1}$ is a bounded operator in $H^s$, in particularly uniformly bounded with respect to $c$, we easily derive the following result by a standard fix point argument combined with \eqref{eq:lips}:

\begin{proposition}
\label{prop21}
Let $s > d/2$ and $M > 0$ a given constant. Then there exist constants $T_s>0$ and $B_s >0$ such that for all $c > 0$ and for all $\psib\in H^s \times H^s$ such that $\Norm{\psib}{H^s} \leq M$,  there exists a mild solution $t \mapsto \wb(t) \in \mathcal{C}^1((0,T_s), H^s \times H^s)$ to the equation \eqref{eq:NLF}, and such that 
$$
\forall\, t \in (0,T_s), \quad \Norm{\wb(t)}{H^s} \leq B_s. 
$$
\end{proposition}

In terms of $u$ and $v$, the previous proposition can be reformulated as

\begin{corollary}
\label{cor21}
Let $s > d/2$, and assume that $\varphi \in H^s$ and $\gamma\in H^s$ are given. Then there exist constants $T_s>0$ and $B_s >0$ such that for all $c$ there exists a mild solution $(u(t), v(t)) \in \mathcal{C}^1((0,T_s), H^s \times H^s)$ to the equation \eqref{eq:NLSc} satisfying the initial boundary conditions \eqref{eq:BCc} and such that 
$$
\forall\, t \in (0,T_s), \quad \Norm{u(t)}{H^s} + \Norm{v(t)}{H^s} \leq B_s. 
$$

\end{corollary}
\begin{proof}
Using the expression \eqref{eq:rlb} for $\psib$ in terms of $\varphi$ and $\gamma$, and the fact that $c \nab^{-1}$ is uniformly bounded in with respect to $c$, we obtain that $\Norm{\psib}{H^s} \leq M$, a constant independent of $c$. This shows that Proposition \ref{prop21} applies. 
\end{proof}
This corollary immediately gives the existence of the exact solution $z(t)$ of \eqref{eq:kgr} on a time interval $(0,T_s)$ and with $H^s$-bounds both independent of $c$. 

\section{Asymptotic expansion}\label{sec:asymp}

On the $c$-independent time interval $(0,T_s)$ given by the previous results, we search $w$ (at least formally) under the form 
$$
\wb(t,x) = \Wb(t,c^2 t,x) = \begin{pmatrix} 
\Us(t,c^2 t,x) \\[1ex] \Vs(t,c^2t,x)
\end{pmatrix}, 
$$
where we separate the high oscillations from the slow time dependency of the solution. 
Here,  $\Wb(t,\theta,x)$ is a function defined on $\R \times \T \times \T^d$ and taking values in $\Cb^2$. 
The equation for $\Wb$ is the transport equation (see \eqref{eq:NLF})
\begin{equation}
\label{eq:trans}
i \partial_t \Wb + i c^2 \partial_\theta \Wb = -c\nab \Wb + c\nab^{-1} \Fb(\Wb), 
\end{equation}
with boundary condition 
\begin{equation}
\label{eq:psiinit}
\Wb(0,0,x) = \psib(x). 
\end{equation}
In a first step, we expand the previous equations with respect to the small parameter $c^{-2}$ in the following sense: 
For a given $k$, we can write 
\begin{eqnarray}
\label{eq:alphn}
c\sqrt{k^2 + c^2} = c^2\sqrt{ 1 + \frac{|k|^2}{c^2}} &=& c^2 + \frac12 |k|^2 + \sum_{n \geq 2} \alpha_n c^{2-2n} |k|^{2n}, \nonumber\\
&=&  c^2 + \frac12 |k|^2 + \sum_{n \geq 1} \alpha_{n+1} c^{-2n} |k|^{2n +2}, 
\end{eqnarray}
for some coefficients $\alpha_n$ defining a power series of convergent radius $1$. 
However, as $k$ is an arbitrary integer, the meaning of the previous expansion in terms of operators has to be understood in the sense of asymptotic expansions; 
this means that for a given sufficiently smooth function $\Wb$, we can write 
$$
c \nab \Wb = c^2 \Wb - \frac12 \Delta \Wb + \sum_{n \geq 1}  \alpha_{n+1} c^{-2n} (-\Delta)^{n+1} \Wb, 
$$
in the sense that for all $N \in \N$, there exists a constant $C_N$ such that  if $\wb \in H^{s + 2N + 4} \times H^{s + 2N +4}$ for some $s \in \N\backslash\{0\}$, then  for all $c > 0$, 
$$ 
\Norm{c \nab \wb - c^2 \wb +\frac12 \Delta \wb - \sum_{n \geq 1}^{N} \alpha_{n+1} c^{-2n} \Delta^{n+1} \wb}{H^s} \leq C_N c^{-2N -2} \Norm{\wb}{H^{s + 2N + 4}}. 
$$
This last equation can be rigorously proved by using Taylor expansions of the function $x \mapsto \sqrt{1 + x}$, $x > 0$ and Fourier decomposition of $\Wb$. 
Similarly, we can write 
$$
c \nab^{-1} = \Big(1 - \frac{\Delta}{c^2}\Big)^{-1/2} = 1 + \frac{1}{2c^2} \Delta + \sum_{n \geq 2} \beta_n c^{-2n} (-\Delta)^{n}, 
$$
for some coefficients $\beta_n$. Hence in view of \eqref{eq:rlb} the initial condition $\psib$ can be expanded in the sense of asymptotic expansion as 
$$
\psib = \sum_{n \geq 0} c^{-2n} \psib_n(x), 
$$ 
with 
\begin{equation}
\label{eq:psin}
\psib_0 = \begin{pmatrix}
\varphi - i \gamma \\[1ex]
\overline{\varphi} - i \overline{\gamma} 
\end{pmatrix},  \quad
\psib_1 =  -\textstyle\frac{i}{2} \Delta \begin{pmatrix}
  \gamma \\[1ex]
  \overline{\gamma}
\end{pmatrix}, \quad
\mbox{and}\quad 
\psib_n = \beta_n(-\Delta)^{n} 
\begin{pmatrix}
  \gamma \\[1ex]
  \overline{\gamma}
\end{pmatrix}. 
\end{equation}

Similarly, Equation \eqref{eq:trans} can be expanded in the sense of asymptotic expansion for smooth $\Wb$ as 
\begin{equation*}
\begin{aligned}
 c^2(i \partial_\theta   + 1) \Wb =& - i \partial_t \Wb +\textstyle \frac12 \Delta \Wb  +  \Fb(\Wb)  +\displaystyle \sum_{n\geq 1} c^{-2n} \Fb_n(\Wb),
\end{aligned}
\end{equation*}
where 
$$
\Fb_n(\Wb) = \alpha_{n+1} (-\Delta)^{n+1} \Wb + \beta_n (-\Delta)^n \Fb(\Wb). 
$$
For later use we will give the first term $\Fb_1(\Wb)$: For $\Wb = (\Us,\Vs)^T$, 
\begin{equation}\label{eq:F1}
\Fb_1(\Wb)= \begin{pmatrix}
\textstyle \frac18 \Delta ^2 \Us + \textstyle \frac12 \Delta f(\textstyle\frac12(\Us+\overline{\Vs}))\\[1ex]
\textstyle \frac18 \Delta ^2 \Vs + \textstyle \frac12 \Delta f(\textstyle\frac12(\Vs+\overline{\Us}))
\end{pmatrix}.
\end{equation}
Note that for $s > d/2$, can prove the following estimates:  
$$
\forall\, n \geq 1, \quad \Norm{\Fb_n(\Wb)}{H^{s}} \leq C_n \Norm{\Wb}{H^{s+ 2 + 2n}}^r, 
$$
where $r$ is the polynomial degree of $f$ (see \eqref{eq:bil1}), and for some constant $C_n$ depending on $n$ and $s$. 
We now look for an asymptotic expansion for $\Wb$ as
$$
\Wb(t,\theta,x) = \sum_{n \geq 0} c^{-2n} \Wb_n (t,\theta,x). 
$$
With this expression, we can  expand the nonlinear term $\Fb(\Wb)$ using the fact that $f$ is polynomial. We obtain an expansion of the form 
$$
\Fb(\Wb) = \Fb(\Wb_0) + \sum_{n \geq 1} c^{-2n} \dd \Fb(\Wb_0) \cdot \Wb_n + \Sb_n(\Wb_j \, |\, j \leq n-1),
$$
where $\Sb_n(\Wb_j \, |\, j \leq n-1)$ denotes a nonlinear term depending only on $\Wb_j$ for $j \leq n-1$. Also note that this term is polynomial in the functions $\Wb_j$. Here,  $\dd \Fb(\Wb_0)$ is the differential of $\Fb$ at  $\Wb_0$. It is a polynomial term of oder $r-1$, in $\Wb_0$ if $f$ is polynomial of degree $r$. 

Hence we get a recurrence relation of the form 
\begin{equation}
\label{eq:induc0}
(i \partial_\theta +1)\Wb_0  = 0, \quad \mbox{and}\quad (i \partial_\theta +1)\Wb_1    =  (- i \partial_t  + \textstyle \frac12\Delta) \Wb_{0}  +  \Fb(\Wb_{0}), 
\end{equation}
and for $n \geq 2$, 
\begin{multline}
\label{eq:induc}
(i \partial_\theta +1)\Wb_n    =  (- i \partial_t  + \textstyle \frac12\Delta) \Wb_{n-1}  +  \dd \Fb(\Wb_{0}) \cdot \Wb_{n-1} + \Rb_n(\Wb_j \, |\, j \leq n-2), 
\end{multline}
where $\Rb_n(\Wb_j \, |\, j \leq n-2)$ denotes a term depending only on $\Wb_j$ for $j \leq n-2$ that is polynomial in the functions $\Wb_j$ and their derivatives. 
Using the previous estimates we obtain that for all $n$ and all collections $\Wb_j$, $j = 0,\ldots, n-2$ of sufficiently smooth functions, 
\begin{equation}
\label{eq:derivW}
\Norm{\Rb_n(\Wb_j \, |\, j \leq n-2)}{H^s} \leq C_n \big(\max_{j = 1,\ldots,n-2}\Norm{\Wb_j}{H^{s + 2(n-j)}}\big)^r , 
\end{equation}
where the constant $C_n$ does not depend on the $\Wb_j$'s (but depends on $n$ and $s$). Eventually, the initial condition \eqref{eq:psiinit} yields the conditions 
\begin{equation}
\label{eq:psibn}
\forall\, n \geq 0, \quad \Wb_n(0,0,x) = \psib_n(x),
\end{equation}
where $\psib_n$ is given by the expansion \eqref{eq:psin}. 

To solve the previous induction relation \eqref{eq:induc} and \eqref{eq:psibn}, we impose the initial condition $\Wb_j = 0$ for $j < 0$. We also introduce the following space: we say that $\Gb(t,\theta, x) \in \Hc_{T}^s$ if $\Gb$ is a trigonometric polynomial in $\theta$ with coefficients in $\mathcal{C}^{1}((0,T), H^s \times H^s)$. In other words, we can write $G$ as
$$
\Gb(t,\theta,x) = \sum_{a\in \Z, \, |a| \leq p } e^{ia\theta} \gb^{(a)}(t,x),
$$
where $p < + \infty$ and 
with $\gb^{(a)}(t,x) \in \mathcal{C}^{1}((0,T), H^s \times H^s)$. By construction, as the nonlinearity is assumed to be polynomial, it is clear that if for $j = 0,\ldots,n-1$, $\Wb_j$, is an element of $\Hc_{T}^{s + 2 (n-j)}$, then we have (see \eqref{eq:derivW}), 
$$
\Rb_n(\Wb_j \, |\, j \leq n-2)\in \Hc_{T}^{s}. 
$$

We will use repeatedly the following result: 
\begin{lemma}[Solution of the homological equation]
Let $T \in \R$ and $s > 0$.  
Assume that $\Gb(t,\theta,x) \in \Hc_{T}^s$ is given.
Then there exists a unique solution $\Wb_*(t,\theta,x) \in  \Hc_{T}^s$ satisfying the equations
\begin{equation}
\label{eq:homo}
(i \partial_\theta +1)\Wb_*(t,\theta,x)    = \Gb(t,\theta,x) - \frac{1}{2\pi}\int_{\T} e^{-i \theta} \Gb(t,\theta,x) \dd \theta , 
\end{equation}
and
\begin{equation}
\label{eq:train}
\forall\, t \in (0,T), \quad \forall x \in \T^d, \quad \frac{1}{2\pi}\int_{\T} e^{-i \theta} \Wb_*(t,\theta,x) \dd \theta  = 0. 
\end{equation}
Moreover, for all function $\wb(t,x) \in \mathcal{C}^1((0,T), H^s \times H^s)$, the function 
$$
\Wb(t,\theta,x) = e^{i\theta } \wb(t,x) + \Wb_*(t,x). 
$$
satisfies the equation \eqref{eq:homo}. 
\end{lemma} 
\begin{proof}
We decompose $\Gb(t,\theta,x)$ in Fourier in $\theta$: we can express this function as 
$$
\Gb(t,\theta,x) = \sum_{|a| \leq p} e^{i a\theta} \gb^{(a)} (t,x), 
$$
with $p < + \infty$ and 
where $\gb^{(a)}(t,x) \in \mathcal{C}^1((0,T),H^s \times H^s)$. Note that 
$$
\frac{1}{2\pi}\int_{\T} e^{-i \theta} \Gb(t,\theta,x) \dd \theta  = \gb^{(1)}(t,x). 
$$
Let us seek $\Wb_*$ under the form
$$
\Wb_{*}(t,\theta,x) = \sum_{|a| \leq p} e^{i a \theta} \wb_*^{(a)}(t,x), 
$$
for some unknown functions $\wb_*^{(a)}$ in $\mathcal{C}^1((0,T),H^s \times H^s)$. 
Equation \eqref{eq:homo} can be written as
$$
\forall\, |a| \leq p , \quad
(1 - a) \wb_*^{(a)}(t,x) = \gb^{(a)}(t,x) - \delta_a^1 \gb^{(1)}(t,x), 
$$
where $\delta_a^1$ denotes the Kronecker symbol. 
It is clear that this equation is solvable for the functions $\wb_*^{(a)}(t,x)$, $a \neq 1$ and that with the condition \eqref{eq:train} we can set $\wb_*^{(1)}(t,x) = 0$. 
The last statement is a consequence of the fact that the kernel of the operator $i \partial_\theta + 1$ is the space spanned by $e^{i\theta}$. 
\end{proof}

We are now ready to state the main result of this section: 

\begin{theorem}\label{thm:main}
Let $N \in \N$, $s > d/2$, $c_0 > 0$ and 
assume that $\varphi\in H^{s + 2N + 4}$ and $\gamma \in  H^{s + 2N + 4}$ are given.  Then there exists $T = T(s,N) > 0$, and a sequence of  functions $\Wb_n(t,\theta,x) \in \Hc_{T}^{s + 2(N-n)}$, satisfying the following properties: $\Wb_0 = \wb_0 e^{i \theta}$ where $\wb_0$ satisfies the equation 
\begin{equation}
\label{eq:nlsw0}
i \partial_t \wb_0 =\textstyle \frac12 \Delta \wb_0 + \langle \Fb\rangle(\wb_0), \quad \wb_0(0) = \psib_0, 
\end{equation}
where $\psib_0$ is defined in \eqref{eq:psin} and where by definition, 
$$
\langle \Fb\rangle(\wb) = \frac{1}{2\pi} \int_0^{2\pi} 
\begin{pmatrix}
f(u + e^{-2i\theta}\overline {v}) \\�f(v + e^{-2 i \theta} \overline u)
\end{pmatrix}  \dd \theta, \quad \mbox{if} \quad \wb = \begin{pmatrix}
u \\�v \end{pmatrix}, 
$$
and for all $n$, $\Wb_n$ solve the equations \eqref{eq:induc}-\eqref{eq:psibn} for $n = 0,\ldots, N$, and can be decomposed into trigonometric polynomials of the form 
\begin{equation}
\label{eq:Wbn}
\Wb_n = \sum_{|a| \leq p_n} e^{ia \theta} \wb_n^{(a)}(t,x),
\end{equation}
with $p_n < +\infty$, and 
where $\wb_n^{(a)}(t,x) \in \mathcal{C}((0,T), H^{s + 2(n-N)} \times  H^{s + 2(n-N)} )$. 
Moreover, there exists a constant $C_N$ such that for $c > c_0$, if  $\wb(t)$ denotes the solution given by Proposition \ref{prop21} with initial value $\psib$ given by \eqref{eq:rlb} in terms of $\varphi$ and $\gamma$, then 
\begin{equation}
\label{eq:asympt}
\forall\, t \leq T, \quad 
\Norm{\wb(t,x) - \sum_{n = 0}^N  c^{-2n} \Wb_n (t,c^{-2}t, x)}{H^{s}} \leq C_N c^{-2 N - 2}. 
\end{equation}

\end{theorem}
\begin{proof}
Let us consider the recursion \eqref{eq:induc} together with the initial condition \eqref{eq:psibn} where the $\psib_n$ are given by \eqref{eq:psin}. For $n = 0$, we obtain the equation 

$$
i (\partial_\theta + 1) \Wb_0(t,\theta,x) = 0,\quad \mbox{and}\quad\Wb(0,0,x) = \psib_0(x)
$$
which implies, according to the previous lemma, that 
\begin{equation}
\label{eq:W0}
\Wb_0(t,\theta,x) = e^{i\theta}\wb_0^{(1)}(t,x), \quad \mbox{and}\quad\wb_0^{(1)}(0,x) = \psib_0(x)
\end{equation}
that is we have determined $\wb_0^{(a)} = 0$ for $a \neq 1$ (and set $p_0 = 1$). 

For $ n = 1$, the equation \eqref{eq:induc0} is written 
\begin{equation}\label{eq:W1}
i (\partial_\theta + 1) \Wb_1(t,\theta,x)  =  (- i \partial_t  + \textstyle \frac12\Delta) \Wb_{0}  +  \Fb(\Wb_{0}). 
\end{equation}
According to the result of given in the previous lemma (see \eqref{eq:homo}), this equation will be solvable if the right-hand side satisfies 
$$
\int_{\T} (- i \partial_t  + \textstyle \frac12\Delta) e^{-i \theta} \Wb_{0}(t,\theta,x)  +  e^{-i \theta}\Fb(\Wb_{0}(t,\theta,x)) \dd \theta = 0. 
$$
In view of \eqref{eq:W0}, this equation can be written as
$$
( i \partial_t  - \textstyle \frac12\Delta)  \wb_{0}^{(1)}(t,x) = \displaystyle \frac{1}{2\pi} \int_{\T} e^{-i \theta}\Fb(e^{i\theta}\wb_{0}^{(1)}(t,x))  \dd \theta,
$$
which gives the equation \eqref{eq:nlsw0} with the boundary condition \eqref{eq:psibn}. Now let $T$ be such that $\wb_0(t,x)$ is a solution to the previous equation with initial value $\wb(0,x) = \psib_0(x)$ in the space $\mathcal{C}^1((0,T), H^{s + 2N +4} \times H^{s +2N +4})$. Then we can set $\wb_{0}^{(1)}(t,x) = \wb_0(t,x)$ and solve for the functions $\wb_1^{(a)}$, $a \neq 1$ according to the previous lemma.  Note that $p_1$ is given by the degree of the polynomial in the nonlinear right-hand side. 

Now assume that $n \geq 1$ is given, and that $\wb_k^{(a)}(t,\theta,x)$, $a \leq p_k$ are known functions in $\mathcal{C}^1((0,T, H^{s + 2(N - k) +4}\times H^{s + 2(N-k) +4})$ for $k \leq n-2$  and $p_k < \infty$, and  that $\wb_{n-1}^{(a)}(t,\theta,x)$ is determined for $a \neq 1$. 

Consider Equation \eqref{eq:induc}. Using \eqref{eq:derivW} and the induction hypothesis, the right-hand side of this equation belongs to $\mathcal{C}^1((0,T, H^{s + 2(N - n)+2}\times H^{s + 2(N-n) +2})$, and is a polynomial in $e^{i\theta}$ of degree $p_n \geq p_{n-1}$, because the nonlinearity is a polynomial. To be solvable, the right-hand side of \eqref{eq:induc} has to be orthogonal to $e^{-i \theta}$ in $L^2(\T)$. 
By induction hypothesis, only the term $w_{n-1}^{(1)}$ has to determined. In view of \eqref{eq:induc}, we thus see that $w_{n-1}^{(1)}$ has to satisfy a linear equation of the form 
\begin{equation}
\label{eq:eqlin}
(- i \partial_t  + \textstyle \frac12\Delta)  \wb_{n-1}^{(1)}(t,x) = \mb_{n-1}(t,x) \cdot \wb_{n-1}^{(1)} + \rb_{n-1}(t,x), 
\end{equation}
where $\mb_{n-1}(t,x)$ is a linear multiplicative operator with coefficients in the space $\mathcal{C}^{1}((0,T), H^{s + 4N+2})$, and where 
$\rb(t,x) \in \mathcal{C}^1((0,T, H^{s + 2(N - n)+2}\times H^{s + 2(N-n)+2})$. The initial condition associated with this equation is written, according to \eqref{eq:psibn}, as
$$
\wb_{n-1}^{(1)}(0,x) = \psib_n(x) + \sum_{|a| \leq p_{n-1}, \, \, a \neq 1} \wb_{n-1}^{(a)}(0,x), 
$$
where $\psib_n(x)$ is given by \eqref{eq:psin} in terms of $\varphi$ and $\gamma$, and where the other terms in the right-hand side are known by induction hypothesis. 

As the equation \eqref{eq:eqlin} in linear with coefficients bounded on the time interval $(0,T)$; 
we deduce that $\wb_{n-1}^{(1)}(t,x)$ is well defined in $\mathcal{C}^1((0,T, H^{s + 2(N - n)+2}\times H^{s + 2(N-n)+2})$.  Using the previous Lemma, we can then solve for the terms $\wb_{n}^{(a)}$ for $a \neq 1$ in $\mathcal{C}^1((0,T, H^{s + 2(N - n)+2}\times H^{s + 2(N-n)+2})$.  

To prove the asymptotic expansion, we note that by construction the term 
$$
\wb^{N}(t,x) :=  \sum_{n = 0}^{N}  c^{-2n} \Wb_n (t,c^{-2}t, x)
$$
belong to $\mathcal{C}^1((0,T, H^{s +2}\times H^{s +2})$ and
satisfies the equation \eqref{eq:NLF} up to an error or order $c^{-2(N+1)}$, that is 
\begin{multline}
\label{eq:approx21}
\forall\, t \in (0,T),\\ 
\begin{array}{rcl}
i \partial_t \wb^N(t,x) &=& - c \nab \wb^N(t,x) + c \nab^{-1} \Fb(\wb^N(t,x)) + c^{-2N -2}\rb_c^N(t,x),\\[1ex]
\wb^N(0,x) &=& \psib(x) + c^{-2N - 2} \psib_c^N
\end{array}
\end{multline}
with
$$
\rb^N_c(t,x) \in \mathcal{C}^1((0,T, H^{s}\times H^{s}), 
\quad \mbox{and}\quad
\psib^N_c(x) \in H^s \times H^s, 
$$
satisfying the following uniform bounds: for a given $c_0 >0$, 
 there exists a constant $C_N$ such that 
$$
\forall\, c > c_0,\quad \forall t \in (0,T), \quad \Norm{\wb^N(t,x) }{H^s} + \Norm{\rb_c^N(t,x)}{H^s} + \Norm{\psib_c^N(x)}{H^s} \leq C_N. 
$$
Using Proposition \ref{prop21}, we can assume that the exact solution $\wb(t)$ is also uniformly bounded (with respect to $c$) on $(0,T)$ in $H^s$. We conclude by using \eqref{eq:approx21} and standard comparison arguments based on the Gronwall Lemma on $(0,T)$ and in $H^s$ with $s > d/2$. 
\end{proof}

\section{First terms}\label{sec:firstt}

The previous result implies the existence of an asymptotic expansion of $u(t,x)$ and $v(t,x)$ solutions of \eqref{eq:NLSc}-\eqref{eq:BCc} up to any arbitrary order, provided the initial condition is smooth. Note that we will use the following notations: for all $n \geq 0$, $\Wb_n = (\Us_n,\Vs_n)^T$ with, according to the decomposition \eqref{eq:Wbn}
\begin{equation}
\label{eq:devU}
\Us_n =  \sum_{|a| \leq p_n} e^{ia \theta} u_n^{(a)}(t,x),\quad \mbox{and}\quad 
\Vs_n =  \sum_{|a| \leq p_n} e^{ia \theta} v_n^{(a)}(t,x). 
\end{equation}
This implies an asymptotic expansion of the solution $u(t)$, $v(t)$ and $z(t)$ of the form 
$$
u(t,x) = \sum_{n \geq 0} c^{-2n} u_n(t,c^2t,x) , \quad 
v(t,x) = \sum_{n \geq 0} c^{-2n} v_n(t,c^2t,x) , 
$$
and
$$
z(t,x) = \sum_{n \geq 0} c^{-2n} z_n(t,c^2t,x), 
$$
with the relation $z_n = \frac12 (u_n + \overline{v_n})$. Note that here, the equality signs must be understood in the sense of asymptotic expansion. Remark also that in the case where $\varphi$ and $\gamma$ are real, that is when the solution $z(t)$ is real, then we have for all $n$, $u_n = v_n$. 
Using the result of Theorem \ref{thm:main}, $z_0(t,x)$ is an approximation of the solution $z(t,x)$ over the time interval $(0,T)$ with an error order $\mathcal{O}(c^{-2})$, while 
\begin{equation}
\label{eq:devz}
z_0(t,x) + \textstyle\frac{1}{c^2} z_1(t,x)
\end{equation}
approximates the exact solution $z(t,x)$ over the time interval $(0,T)$ up to terms of order $c^{-4}$.
The goal of this section is to express the two first terms $z_0$ and $z_1$.  For simplicity we consider the one dimensional case and assume the initial values $\varphi, \gamma$ to be sufficiently smooth so that Theorem \ref{thm:main} can be applied. Note that a generalization to higher dimensions follows the line of argumentation. We will use these computations to derive our numerical schemes in Section \ref{sec:nume}. For reasons of ease we will focus on the linear case $f(z) = \lambda z$, see Section \ref{ex:lin}, and a cubic-nonlinearity $f(z) = \lambda|z|^2 z$, see Section \ref{ex:cub}. Nevertheless we state the equations for $\Us_1$ and $\Vs_1$ in the general case $f(z) = \lambda|z|^{2p} z$,  $p\in\N$. 

The first  term $z_0$ approximating the exact solution of Klein-Gordon equation \eqref{eq:kgr} up to terms of order $c^{-2}$ is easily derived from Theorem \ref{thm:main} using \eqref{eq:zuv} and \eqref{eq:rlb}. More precisely, we have 
$$
z_{0}(t,x) = \textstyle \frac{1}{2} e^{itc^2} u_0(t,x) + \frac12 e^{-itc^2} \overline{v_0}(t,x),
$$
where $u_0, v_0$ solve the nonlinear Schr\"odinger system \eqref{eq:nlsw0}. 

In order to explicit the term $z_1(t,x)$,  we expand the non-linearity $f(\frac12 (\Us +\overline{\Vs}))$ in a Taylor-series, up to terms of order $c^{-2}$. With the notation $\Zs_k = \frac12 (\Us_k +\overline{\Vs_k})$ and keeping in mind that $f(z) = f(z,\overline{z}) = \lambda \vert z\vert^{2p} z = \lambda z^{p+1} \overline{z}^p$, we obtain
\begin{equation*}
\begin{aligned}
f(\textstyle\frac12 (\Us +\overline{\Vs})) &= f(\mathsf{Z}_0+\textstyle\frac{1}{c^2} \mathsf{Z}_1)+\mathcal{O}(c^{-4}) \\&= f(\mathsf{Z}_0) + \textstyle\frac{1}{c^2}\big(\partial_z f(\mathsf{Z}_0,\overline{\mathsf{Z}_0}) \mathsf{Z}_1 + \partial_{\overline{z}} f(\mathsf{Z}_0,\overline{\mathsf{Z}_0})\overline{\mathsf{Z}_1}\big) +\mathcal{O}(c^{-4})\\
&= f(\mathsf{Z}_0) + \textstyle\frac{\lambda}{c^2} \left[ (p+1) \vert \mathsf{Z}_0\vert^{2p}\mathsf{Z}_1+ p  \vert \mathsf{Z}_0\vert^{2p-2}(\mathsf{Z}_0)^2 \overline{\mathsf{Z}_1}\right]+\mathcal{O}(c^{-4}). 
\end{aligned}
\end{equation*}
Furthermore, in the above notation we calculate that 
\begin{equation*}
\begin{aligned}
& \Delta f = \partial_{zz} f (\partial_x \mathsf{Z})^2 + 2\partial_{z\overline{z}}f(\partial_x \mathsf{Z})(\partial_x \overline{\mathsf{Z}})+\partial_z f (\partial_{xx} \mathsf{Z}) +\partial_{\overline{z}\overline{z}} f (\partial_x \overline{\mathsf{Z}})^2 + \partial_{\overline{z}} f (\partial_{xx} \overline{\mathsf{Z}}), 
\end{aligned}
\end{equation*}
where in this formula $f$ is evaluated in $(\mathsf{Z},\overline{\mathsf{Z}})$. 

Thus, we obtain by \eqref{eq:induc} with $n=2$ (using \eqref{eq:F1} and the above calculations) that the homological equation in terms of $\mathsf{U}_2$ is written 
\begin{equation}\label{eq:U2}
\begin{aligned}
(i \partial_\theta  +1)\mathsf{U}_2 =& -i\partial_t \mathsf{U}_1 + \textstyle \frac12 \Delta \mathsf{U}_1 + \textstyle \frac18 \Delta ^2 \mathsf{U}_0 \\&+\lambda (p+1) \vert \textstyle \frac12 (\mathsf{U}_0 + \overline{\mathsf{V}_0}) \vert^{2p} \textstyle \frac12(\mathsf{U}_1+\overline{\mathsf{V}_1}) \\
&+ \lambda p\vert  \textstyle \frac12(\mathsf{U}_0 + \overline{\mathsf{V}_0}) \vert^{2p-2}( \textstyle \frac12(\mathsf{U}_0+\overline{\mathsf{V}_0}))^{2}\textstyle \frac12(\mathsf{V}_1+\overline{\mathsf{U}_1})
\\&+\lambda (p+1) \vert \textstyle \frac12 (\mathsf{U}_0 + \overline{\mathsf{V}_0}) \vert^{2p}\textstyle \frac14\Delta (\mathsf{U}_0+\overline{\mathsf{V}_0})\\
& + \lambda p\vert  \textstyle \frac12(\mathsf{U}_0 + \overline{\mathsf{V}_0}) \vert^{2p-2}( \textstyle \frac12(\mathsf{U}_0+\overline{\mathsf{V}_0}))^{2} \textstyle \frac14\Delta (\mathsf{V}_0+\overline{\mathsf{U}_0})
\\
&+\lambda \textstyle \frac12p(p+1)\vert \textstyle \frac12 (\mathsf{U}_0 + \overline{\mathsf{V}_0}) \vert^{2p-2}\frac12(\overline{\mathsf{U}_0}+\mathsf{V}_0)\big(\partial_x  \frac12 (\mathsf{U}_0 + \overline{\mathsf{V}_0}) )^2\\
&+\lambda\textstyle p(p+1)\vert\textstyle  \frac12 (\mathsf{U}_0+ \overline{\mathsf{V}_0}) \vert^{2p-2}\textstyle  \frac12 (\mathsf{U}_0 + \overline{\mathsf{V}_0}) \vert\partial_x  \frac12 (\mathsf{U}_0 + \overline{\mathsf{V}_0})\vert^2\\
&+\lambda \textstyle \frac12p(p-1)\vert\textstyle \frac12 (\mathsf{U}_0 + \overline{\mathsf{V}_0}) \vert^{2p-4}( \frac12 (\mathsf{U}_0 + \overline{\mathsf{V}_0}) )^3\big(\partial_x \frac12 (\overline{\mathsf{U}_0}+\mathsf{V}_0)\big),
\end{aligned}
\end{equation}
and a similar equation for $\Vs_2$ with $\Us_k$ replaced by $\Vs_k$, $k=0,1,2$, and vice versa.
\subsection{Formula for the first terms in the linear case}\label{ex:lin}

In the linear case we have $f(z) =\lambda z$. In this case we obtain the following limit equation for $u_0$ and $v_0$
\begin{equation}\label{eq:u0lin}
\begin{aligned}
\textstyle(- i\partial_t+\frac12 \Delta + \frac{\lambda}{2} )u_0 &= 0,\qquad u_0(0) = \varphi - i \gamma,\\
\textstyle(- i\partial_t+\frac12 \Delta + \frac{\lambda}{2}) v_0 &= 0,\qquad v_0(0) = \overline{\varphi+i\gamma},
\end{aligned}
\end{equation}
see \eqref{eq:nlsw0}. For $\Us_1$ we obtain by \eqref{eq:W1}
\begin{equation*}
(i\partial_\theta +1)\Us_1 = (-i\partial_t +\textstyle \frac12\Delta)u_0\mathrm{e}^{i\theta} + \frac12\lambda(u_0\mathrm{e}^{i\theta}+\overline{v_0}\mathrm{e}^{-i\theta}).
\end{equation*}
Thus, plugging the Fourier expansion \eqref{eq:devU} of $\Us_1$ in $\theta$, we obtain with the orthogonality relation \eqref{eq:u0lin} for $u_0$ that
\begin{equation*}
\sum_{a\in\Z} (1-a) u_1^{(a)}\mathrm{e}^{ia\theta} = \textstyle \frac{1}{2}\lambda\overline{v_0}\mathrm{e}^{-i\theta}.
\end{equation*}
Thus, $u_1^{(a)} = 0$ for $a\neq \pm 1$ and $u_{1}^{(-1)} = \textstyle\frac14 \lambda \overline{v_0}$, and we set $\xi_1 := u_1^{(1)}$, so that we can write
\begin{equation}
\Us_1(t,\theta,x)= \textstyle \frac14\lambda \overline{v_0}(t,x) \mathrm{e}^{-i\theta}+\xi_1(t,x)\mathrm{e}^{i\theta}, 
\end{equation}
Note that similarly $\Vs_1(t,\theta,x)= \frac14\lambda \overline{u_0}(t,x) \mathrm{e}^{-i\theta}+\eta_1(t,x)\mathrm{e}^{i\theta}$ with $\eta_1 := v_1^{(1)}$. 
In order to determine $\xi_1$ and $\eta_1$  we analyze \eqref{eq:U2} in the linear case, i.e. $p=0$, which reduces to
\begin{equation*}
(i\partial_\theta +1) \Us_2 = -i \partial_t \Us_1 +\textstyle \frac12 \Delta \Us_1 + \frac18 \Delta^2 \Us_0 +\lambda \frac12 (\Us_1+\overline{\Vs_1})+\frac14\lambda \Delta (\Us_0+\overline{\Vs_0}).
\end{equation*}
To be solvable, the right-hand side of the equation needs to be orthogonal to $\mathrm{e}^{-i\theta}$ which yields the following equation for $\xi_1(t,x)$:
\begin{equation*}
\textstyle(- i\partial_t+\frac12 \Delta + \frac{\lambda}{2})\xi_1 = \textstyle -\frac18\Delta^2 u_0 -\frac18\lambda^2 u_0 - \frac14 \lambda \Delta u_0.
\end{equation*}
Using twice the Schr\"odinger equation \eqref{eq:u0lin} for $u_0$, the above equation can be rewritten as
\begin{equation*}
\textstyle(- i\partial_t+\frac12 \Delta + \frac{\lambda}{2})u_1 = \textstyle \frac12 \partial_{tt} u_0.
\end{equation*}
The same equation holds true for $\eta_1$ with $\xi_1$ and $u_0$ replaced by $\eta_1$ and $v_0$, respectively. The initial values $\xi_1(0)$ and $\eta_1(0)$ have to be chosen such that $\Wb_1(0,0,x) = \boldsymbol{\psi}_1(x)$, see \eqref{eq:psibn}. As $\Us_1(0) = \frac{1}{4} \lambda\overline{v_0}(0)+\xi_1(0)$ and $\Vs_1(0) =  \frac{1}{4}\lambda \overline{u_0}(0)+\eta_1(0)$ we thus obtain by \eqref{eq:psin} the following initial value problems for $\xi_1$ and $\eta_1$
\begin{equation}\label{eq:u1lin}
\begin{aligned}
& \left(\textstyle- i\partial_t+\frac12 \Delta + \frac{\lambda}{2}\right) \xi_1 = \textstyle \frac12 \partial_{tt} u_0,\, &\xi_1(0) =   \textstyle \frac14 (\Delta u_0(0)-(\Delta+\lambda) \overline{v_0}(0)),\\
& \left(\textstyle- i\partial_t+\frac12 \Delta + \frac{\lambda}{2}\right) \eta_1 = \textstyle \frac12 \partial_{tt} v_0,\, & \eta_1(0) =   \textstyle \frac14 (\Delta v_0(0)-(\Delta+\lambda) \overline{u_0}(0)).
\end{aligned}
\end{equation}
We summarize our findings in the following corollary.
\begin{corollary}[First terms in the linear case]\label{cor:lin}
In the linear case, i.e. $f(z) = \lambda z$, the first and second  terms in the asymptotic expansion  \eqref{eq:devz} read
\begin{equation*}
z_0(t,x) = \textstyle \frac{1}{2} e^{itc^2} u_0(t,x) + \frac12 e^{-itc^2} \overline{v_0}(t,x)
\end{equation*}
and
\begin{equation}\label{eq:z2lin}
z_1(t,x) =\textstyle \frac\lambda{8} \big(u_0(t,x)\mathrm{e}^{ic^2t}+\overline{v_0}(t,x)\mathrm{e}^{-ic^2t}\big) + \frac{1}{2}\big(\xi_1(t,x)\mathrm{e}^{ic^2t}+\overline{\eta_1}(t,x)\mathrm{e}^{-ic^2t}\big),
\end{equation}
where $u_0$, $v_0$ and $\xi_1$, $\eta_1$ solve \eqref{eq:u0lin} and \eqref{eq:u1lin}, respectively. Then the functions $z_0(t,x)$ and 
$z_0(t,x) + \frac{1}{c^2}z_1(t,x)$ approximate the exact solution of \eqref{eq:kgr} in $H^s$ up to terms of order $c^{-2}$ and $c^{-4}$ for initial values in $H^{s+4}$ and $H^{s+6}$, respectively.
\end{corollary}
Note that both corrections terms \eqref{eq:u0lin} and \eqref{eq:u1lin} can be solved exactly in Fourier. Thus, in the linear case we can state the correction terms $z_0$ and $z_1$ in an exact way. See Section \ref{sec:lin} for further details.

\subsection{Formula for the first terms in the cubic case}\label{ex:cub}

We consider now the case of cubic nonlinearity, i.e. $p = 1$ with the previous notations. 
Moreover, for simplicity we consider the case, where $u_0 \equiv v_0$, i.e. $z(0)$ and  $\partial_t z(0)  \in \R$. Thus, $f(\frac12(\Us_0 + \overline{\Vs_0})) =\textstyle \frac{\lambda}{16}\big (u_0\mathrm{e}^{i\theta}+\overline{u_0}\mathrm{e}^{-i\theta}\big)^3$ and the limit equation for $u_0$ reduces to
\begin{equation}\label{eq:u0rel} 
-i\partial_t u_0 +\textstyle\frac{1}{2}\Delta u_0 + \frac{3}{8} \lambda \vert u_0\vert^2 u_0 = 0,\qquad u_0(0) = \varphi - i \gamma,
\end{equation}
see \eqref{eq:nlsw0}. Plugging the Fourier expansion \eqref{eq:devU} of $\Us_1$ in $\theta$ into \eqref{eq:W1}, we obtain with the orthogonality relation \eqref{eq:u0rel}  for $u_0$ that
$$
\sum_{a\in\Z} (1-a)u_1^{(a)} \mathrm{e}^{ia\theta} =\textstyle\frac{\lambda}{8} \Big((u_0)^3\mathrm{e}^{3i\theta}+3\vert u_0\vert^2\overline{u_0}\mathrm{e}^{-i\theta} + (\overline{u_0})^3\mathrm{e}^{-3i\theta}\Big).
$$
Hence, $u_1^{(a)} = 0$ for $a\neq -3,-1,1,3$ and
$$
u_{1}^{(-3)} = \textstyle\frac{\lambda}{32}(\overline{u_0})^3,\quad u_{1}^{(-1)} = \frac{3\lambda}{16}\vert u_0\vert^2 \overline{u_0},\quad\mbox{and}\quad u_1^{(3)} = -\textstyle\frac{\lambda}{16} (u_0)^3. 
$$
Setting $\xi_1 := u_1^{(1)}$, we obtain 
\begin{equation}\label{eq:U1cub}
\Us_1= \textstyle \frac{\lambda}{8}\Big(- \frac12 (u_0)^3 \mathrm{e}^{3i\theta}+\frac32 \vert u_0\vert^2\overline{u_0} \mathrm{e}^{-i\theta}+\frac14 (\overline{u_0})^3\mathrm{e}^{-3i\theta}\Big) + \xi_1 \mathrm{e}^{i\theta} (= \Vs_1). 
\end{equation}
In order to determine the function $\xi_1$ we need to state the equation for $\Us_2 = \sum_{a\in\Z} u_2^{(a)} \mathrm{e}^{i a \theta}$,  which reduces in our cubic setting to
\begin{equation*}
\begin{aligned}
\sum_{a\in\Z} (1- & a)u_2^{(a)} \mathrm{e}^{ia\theta} =\sum_{a\in \Z} (-i\partial_t + \textstyle\frac12 \Delta) u_1^{(a)} \mathrm{e}^{ia\theta} +\textstyle\frac18\Delta^2 u_0 \mathrm{e}^{i\theta}\\& +\textstyle\frac{3}{16}\lambda\Big((u_0)^2\mathrm{e}^{2i\theta}+2\vert u_0\vert^2+(\overline{u_0})^2\mathrm{e}^{-2i\theta}\Big)(\Delta u_0\mathrm{e}^{i\theta}+\Delta \overline{u_0}\mathrm{e}^{-i\theta})\\&
+\textstyle \frac38 \lambda \Big ((\partial_x u_0)^2\mathrm{e}^{2i\theta}+2\vert \partial_x u_0\vert^2+(\partial_x \overline{u_0})^2\mathrm{e}^{-2i\theta}\Big)(u_0\mathrm{e}^{i\theta}+\overline{u_0}\mathrm{e}^{-i\theta})\\&
+\textstyle \frac38\lambda\Big((u_0)^2\mathrm{e}^{2i\theta}+2\vert u_0\vert^2+(\overline{u_0})^2\mathrm{e}^{-2i\theta}\Big)\sum_{a\in\Z} (u_1^{(a)} \mathrm{e}^{ia\theta}+ \overline{u_1^{(a)}}\mathrm{e}^{-ia\theta}),
\end{aligned}
\end{equation*}
see \eqref{eq:U2}. To be solvable, the right-hand side of the equation needs to be orthogonal to $\mathrm{e}^{-i\theta}$ which yields the following equation for $\xi_1(t,x) = u_1^{(1)}(t,x)$: 
\begin{equation*}
\begin{aligned}
& (i \partial_t - \textstyle \frac12 \Delta )\xi_1 -\lambda \textstyle \frac34 \vert u_0\vert^2 \xi_1 -\lambda \frac38 (u_0)^2\overline{\xi_1}= \textstyle \frac18 \Delta^2 u_0 +\lambda^2 \frac{51}{256}\vert u_0\vert^4 u_0 +\frac{3}{2}\Delta f(\frac12 u_0).
\end{aligned}
\end{equation*}
The initial value $\xi_1(0)$ needs to be chosen such that $\Us_1(0) =- \frac{i}{2}\Delta \gamma$, see \eqref{eq:psibn} as well as the definition \eqref{eq:psin}. Thus, by \eqref{eq:U1cub} we obtain that
\begin{equation*}
\xi_1(0) = \textstyle \frac{\lambda}{16}u_0(0)^3 - \frac{\lambda}{32}\overline{u_0}(0)^3-\frac{3\lambda}{16}\vert u_0(0)\vert^2 \overline{u_0}(0)+\frac14 \Delta (u_0(0)-\overline{u_0}(0)).
\end{equation*}
Using \eqref{eq:u0rel} we rewrite the equation for $u_1$ as
\begin{equation}\label{eq:u1rel}
\begin{aligned}
& (i \partial_t - \textstyle \frac12 \Delta )\xi_1 -\lambda \textstyle \frac34 \vert u_0\vert^2 \xi_1 -\lambda \frac38 (u_0)^2\overline{\xi_1}\\& \qquad\qquad\qquad\qquad\qquad= \textstyle \frac14 i \partial_t \Delta u_0 +\lambda^2 \frac{51}{256}\vert u_0\vert^4 u_0 +\textstyle\lambda \frac{3}{32}\Delta  (\vert u_0\vert^2 u_0),\\
&\xi_1(0) = \textstyle \frac{\lambda}{16}u_0(0)^3 - \frac{\lambda}{32}\overline{u_0}(0)^3-\frac{3\lambda}{16}\vert u_0(0)\vert^2 \overline{u_0}(0)+\frac14 \Delta (u_0(0)-\overline{u_0}(0)).
\end{aligned}
\end{equation}
Again we summarize our findings in the following corollary.
\begin{corollary}[First terms in the cubic case]\label{cor:cub}
For a cubic non-linearity, i.e. $f(z) = \lambda \vert z\vert^2 z$, and for  real initial values functions $z(0,x) = \varphi(x) \in \R$ and $\partial_t z(0,x) = c^2 \gamma(x) \in \R$, the first and second terms in the expansion \eqref{eq:devz} read
\begin{equation*}
z_0(t,x) = \textstyle \frac{1}{2} e^{itc^2} u_0(t,x) + \frac12 e^{-itc^2} \overline{u_0}(t,x),
\end{equation*}
and
\begin{equation}\label{eq:corr2}
\begin{aligned}
z_1(t,x)=& \textstyle \frac{3\lambda}{32}\vert u_0(t,x)\vert^2 (u_0(t,x)\mathrm{e}^{ic^2t}+\overline{u_0}(t,x)\mathrm{e}^{-ic^2t})\\&- \textstyle\frac{\lambda}{64}\big( u_0(t,x)^3\mathrm{e}^{3ic^2t}+\overline{u_0}(t,x)^3\mathrm{e}^{-3ic^2t}\big) \\&+  \textstyle \frac{1}{2} \big( \xi_1(t,x) \mathrm{e}^{ic^2t}+\overline{\xi_1}(t,x)\mathrm{e}^{-ic^2t}\big),
\end{aligned}
\end{equation}
where $u_0$ and $\xi_1$ solve \eqref{eq:u0rel} and \eqref{eq:u1rel}, respectively.  Then the functions $z_0(t,x)$ and 
$z_0(t,x) + \frac{1}{c^2}z_1(t,x)$ approximate the exact solution of \eqref{eq:kgr} in $H^s$ up to terms of order $c^{-2}$ and $c^{-4}$ for initial values in $H^{s+4}$ and $H^{s+6}$, respectively.
\end{corollary}

\section{Energy and charge density conservation}\label{sec:enq}

Analyzing the conserved quantities of the Klein-Gordon equation \eqref{eq:kgr}, we follow Dirac's approach via the ``hole" theory, i.e. antiparticles are related to negative energy eigenstates ($u$ particle with positive charge, $v$ antiparticle with negative charge). This allows us to overcome the problem of \emph{indefinite probability density} (i.e. if we would consider $z(t)$ as a probability amplitude, its probability density $\rho(z) =  \frac{i}{2c^2}\left(\overline{z}\partial_t z - z \partial_t\overline{z} \right)$ which is naturally defined by the continuity equation, would allow positive and negative values.) For further details on the quantization of the Klein-Gordon field we refer to \cite{BjoD65}.

In this interpretation ($\rho(z)$ being the \emph{charge}-density) the conserved physical quantities of the Klein-Gordon equation are the charge
\begin{equation}\label{eq:chargi}
Q(z) =  \frac{i}{2c^2}\int_\T z\partial_t \overline{z}-\overline{z} \partial_t z\,\mathrm{d}x
= \int_\T \mathrm{Re}\left( \frac{-i}{c^2} \partial_t z\overline{z}\right) \,\mathrm{d}x,
\end{equation}
as well as the energy
\begin{equation}\label{eq:eni}
E(z) = \int_\T \vert \textstyle \frac{1}{c} \partial_t z\vert^2 + \vert \nabla z\vert^2 + \vert c z\vert^2 - \textstyle \frac{1}{p+1}\vert z\vert^{2p+2}\,\mathrm{d}x. 
\end{equation}
Hence for all times $t$ where the solution exists and is smooth enough, 
$$
Q(z(t)) = Q(z(0)) = \int_\T \mathrm{Im}(\overline{\varphi}\gamma)\,\mathrm{d}x, 
$$
and
$$
E(z(t)) = E(z(0)) = c^2\big( \Vert \varphi \Vert_{L^2}^2 + \Vert \gamma \Vert_{L^2}^2\big)+ \int_\T \vert \nabla \varphi \vert^2 - \textstyle \frac{1}{p+1}\vert \varphi\vert^{2p+2}\,\mathrm{d}x. 
$$

Using the result given by Theorem \ref{thm:main}, and under smoothness assumptions on $\varphi$ and $\gamma$, we can easily prove that the charge and energy admit asymptotic expansion with respect to $c^{-2}$. In this section, we give the expression of the first terms of these expansions. We commence with the asymptotic expansion of the charge $Q(z)$, i.e. 
$$
 Q(z) = Q_0(\Us_0,\Vs_0) + \sum_{k\geq 1} c^{-2k} Q_k (\Us_0,\ldots,\Us_{k},\Vs_0,\ldots,\Vs_{k}).
$$
Using \eqref{eq:uv} we can easily (as we are working on the torus $\T$ and thus integration by parts yields no boundary-terms) rewrite the charge as
\begin{equation}
\label{eq:quqv}
\begin{aligned}
 Q(z) &= \frac14 \int_\T\mathrm{Re}\left( \big(\frac{\nab}{c} (u-\overline{v})\big)(\overline{u}+v)\right)\,\mathrm{d}x\\
& = \frac14  \int_\T \vert u\vert^2 - \vert v\vert^2 + \frac{1}{2c^2}( \vert \nabla u\vert^2 - \vert \nabla v\vert^2) \,\mathrm{d}x+ \mathcal{O}(c^{-4}).
\end{aligned}
\end{equation}
Thus, using the asymptotic expansion given in Theorem \ref{thm:main}, we obtain
\begin{equation}\label{eq:Q0}
Q_0(\Us_0,\Vs_0) = \frac{1}{4}\left(\Vert u_0\Vert_{L^2}^2 -\Vert v_0\Vert_{L^2}^2\right),
\end{equation}
where $u_0, v_0$ solve the nonlinear Schr\"odinger system \eqref{eq:nlsw0} and hence in particular their $L^2$ norms are conserved. Thus, the first  term of the charge reads
\begin{equation}\label{eq:Q0}
Q_0(\Us_0(t),\Vs_0(t)) = \frac{1}{4}\left(\Vert u_0(0)\Vert_{L^2}^2 -\Vert v_0(0)\Vert^2_{L^2}\right).
\end{equation}
Using the relation $z = \textstyle \frac12 (u+\overline{v})$ as well as $\partial_t z = \textstyle \frac12 i c \nab (u-\overline{v})$, see \eqref{eq:NLSc}, we obtain
$$
E(z) = \frac12 \int_\T  c^2 (\vert u\vert^2 + \vert v\vert^2) + \vert \nabla u\vert^2 + \vert \nabla v\vert^2 - \textstyle \frac{2}{p+1} \vert \textstyle \frac12(u+\overline{v})\vert^{2p+2}\,\mathrm{d}x +  \mathcal{O}(c^{-2}).
$$
Hence similarly to $Q(z)$, we obtain an asymptotic expansion of the form 
$$
E(z) = c^2 E_{-2}(\Us_0,\Vs_0) + \sum_{k\geq 0} c^{-2k} E_k (\Us_0,\ldots,\Us_{k},\Vs_0,\ldots,\Vs_{k}), 
$$
with 
$$
E_{-2}(\Us_0,\Vs_0) = \frac{1}{2}\left(\Vert u_0\Vert_{L^2}^2 + \Vert v_0\Vert_{L^2}^2\right).
$$

In order to analyze the asymptotic behavior of the energy $E(z)$, we first have to subtract the relativistic rest-energy $c^2Q_{\text{mass}}$ (see below) as from the energetic point of view the non-relativistic limit regime means that $\Ec := E - c^2Q_{\text{mass}}  $ is small. 
According to \eqref{eq:quqv}, we can write 
$$
Q(z) = Q^u(u) + Q^v(v),
$$
where in the sense of asymptotic expansion (see \eqref{eq:alphn}), 
$$
Q^u (u) = \frac14  \int_\T \vert u\vert^2+ \frac{1}{2c^2} \vert \nabla u\vert^2 + \sum_{j\geq 2} \alpha_j c^{-2j} \vert \nabla^{j} u\vert \,\mathrm{d}x,
$$
and where $Q^v(v)$ denotes the part of $Q(z)$ which contains the antiparticle component, i.e.
$$
Q^v(v) = Q^u (-v).
$$
With this notation, we define the rest energy $c^2Q_{\text{mass}}(z) = 2 c^2(Q^u(u) - Q^v(v))$. 
Using these notations, we obtain an asymptotic expansion of the form
\begin{equation}\label{eq:asympE}
\begin{aligned}
\Ec(z) &:= E(z) - 2c^2 (Q^u-Q^v)  \\
&\; = \mathcal{E}_0(\Us_0,\Vs_0) + \sum_{k\geq 1} c^{-2k} \Ec_k (\Us_0,\ldots,\Us_{k},\Vs_0,\ldots,\Vs_{k}).
\end{aligned}
\end{equation}
Using the asymptotic expansion given in Theorem \ref{thm:main}, we calculate that the first-order correction term of the energy $\Ec(z)$ in \eqref{eq:asympE} is given by
\begin{equation}\label{eq:E0}
\Ec_0(\Us_0,\Vs_0) = \int_\T \frac14\left (\vert \nabla u_0 \vert^2 + \vert \nabla v_0\vert^2\right) - \frac{1}{p+1} \left \vert \textstyle \frac12(u_0+\overline{v_0}\mathrm{e}^{-2i c^2t})\right\vert^{2p+2}\,\mathrm{d}x,
\end{equation}
where $u_0,v_0$ solve the nonlinear Schr\"odinger system \eqref{eq:nlsw0}.

\section{Construction of robust numerical schemes}\label{sec:nume}
The aim of this section is to show that we can use the derived asymptotic expansions to construct numerical schemes for the non-relativistic Klein-Gordon equation \eqref{eq:kgr} which are robust with respect to the speed of light $c$. Firstly we underline our theoretical results in the linear case, where the exact solution to \eqref{eq:kgr} is analytically known, see Section \ref{sec:lin}. Finally, in Section \ref{sec:cub} we construct robust numerical methods for the cubic non-relativistic Klein-Gordon equation and discuss their convergence. Numerical experiments are included.

\subsection{Linear case}\label{sec:lin}
In the linear case, i.e. $p = 0$, our model problem reads
\begin{equation}\label{eq:kgrlin}
\begin{aligned}
&\frac{1}{c^2} \partial_{tt} z - \Delta z + c^2 z =  \lambda z,\\
& z(0,x) = \varphi=\sum_{a\in\Z}\hat{\varphi}_a \mathrm{e}^{iax}, \quad  \partial_t z(0,x) =c^2\gamma=c^2\sum_{a\in\Z}\hat{\gamma}_a \mathrm{e}^{iax}
\end{aligned}
\end{equation}
with the exact solution
\begin{equation}\label{eq:zlin}
\begin{aligned}
z(t,x) = \sum_{a\in\Z}\Big (\hat{\varphi}_a\cos(c t&\sqrt{a^2+c^2-\lambda})\\&+\frac{c}{\sqrt{a^2+c^2-\lambda}}\hat{\gamma}_a\sin(ct\sqrt{a^2+c^2-\lambda})\Big)\mathrm{e}^{iax}.
\end{aligned}
\end{equation}
In the linear case we can solve the first- and second-order limit systems exactly in Fourier, see Corollary \ref{cor:lin} and equation \eqref{eq:u0lin} and \eqref{eq:u1lin}, respectively. Thus, we can state the first- and second correction terms $z_0$ and $z_1$ in an exact way. A simple computation shows that
\begin{equation}\label{eq:z1lin}
z_0(t,x)= \sum_{a\in\Z}\Big(\textstyle \hat{\varphi}_a \cos\Big(\big(c^2+\frac{1}{2}(a^2-\lambda)\big)t\Big)+\hat{\gamma}_a\sin\Big(\big(c^2+\frac{1}{2}(a^2-\lambda)\big)t\Big)\Big)\mathrm{e}^{iax}.
\end{equation}
Thus, we nicely see that the difference between the exact solution $z$ of the linear Klein-Gordon equation and its first-order correction term $z_0$ in the non-relativistic limit $c\gg 1$ are the formal approximations
\begin{equation*}
 c\sqrt{a^2+c^2-\lambda} \approx c^2+\frac{1}{2}(a^2-\lambda),\quad \frac{c}{\sqrt{a^2+c^2-\lambda}}\approx 1,
\end{equation*}
which are of order $\mathcal{O}(c^{-2})$ for sufficiently smooth initial values $\varphi$ and $\gamma$.

For the second correction term $z_1$ we get after a short computation
\begin{equation}
\begin{aligned}
z_1(t,x) = &\sum_{a\in\Z} \textstyle \Big[\hat{\varphi}_a \frac{(a^2-\lambda)^2t}{8}\sin(\omega_{a,c} t)\\
&+ \textstyle \hat{\gamma}_a\Big( (-\frac{a^2-\lambda}{2})\sin(\omega_{a,c}t)-\frac{(a^2-\lambda)^2t}{8}\cos(\omega_{a,c}t)\Big)\Big]\mathrm{e}^{iax}.
\end{aligned}
\end{equation}
where 
$$
\omega_{a,c} = c^2+\textstyle\frac{1}{2}(a^2-\lambda). 
$$
Again this can easily be motivated by formally taking the next approximations in the exact solution $z$, i.e.
$$
 c\sqrt{a^2+c^2-\lambda} \approx c^2+\frac{1}{2}(a^2-\lambda)-\frac{1}{8c^2}(a^2-\lambda)^2,
 $$
 and
 $$\frac{c}{\sqrt{a^2+c^2-\lambda}}\approx 1 -\frac{a^2-\lambda}{2c^2},
$$
which are of order $\mathcal{O}(c^{-4})$ for sufficiently smooth initial values $\varphi$ and $\gamma$.

The numerical results for the initial values 
$$
\varphi = \frac{(2+i)}{\sqrt{5}}\cos(x), \quad\gamma = \frac{(1+i)}{\sqrt{2}}\sin(x)+\frac{1}{2}\cos(x)
$$
 and $\lambda = -1$ on the time interval $[0,1]$ are given in Figure \ref{figure0}.

\begin{figure}[h!]
\centering
\includegraphics[width=0.52\linewidth]{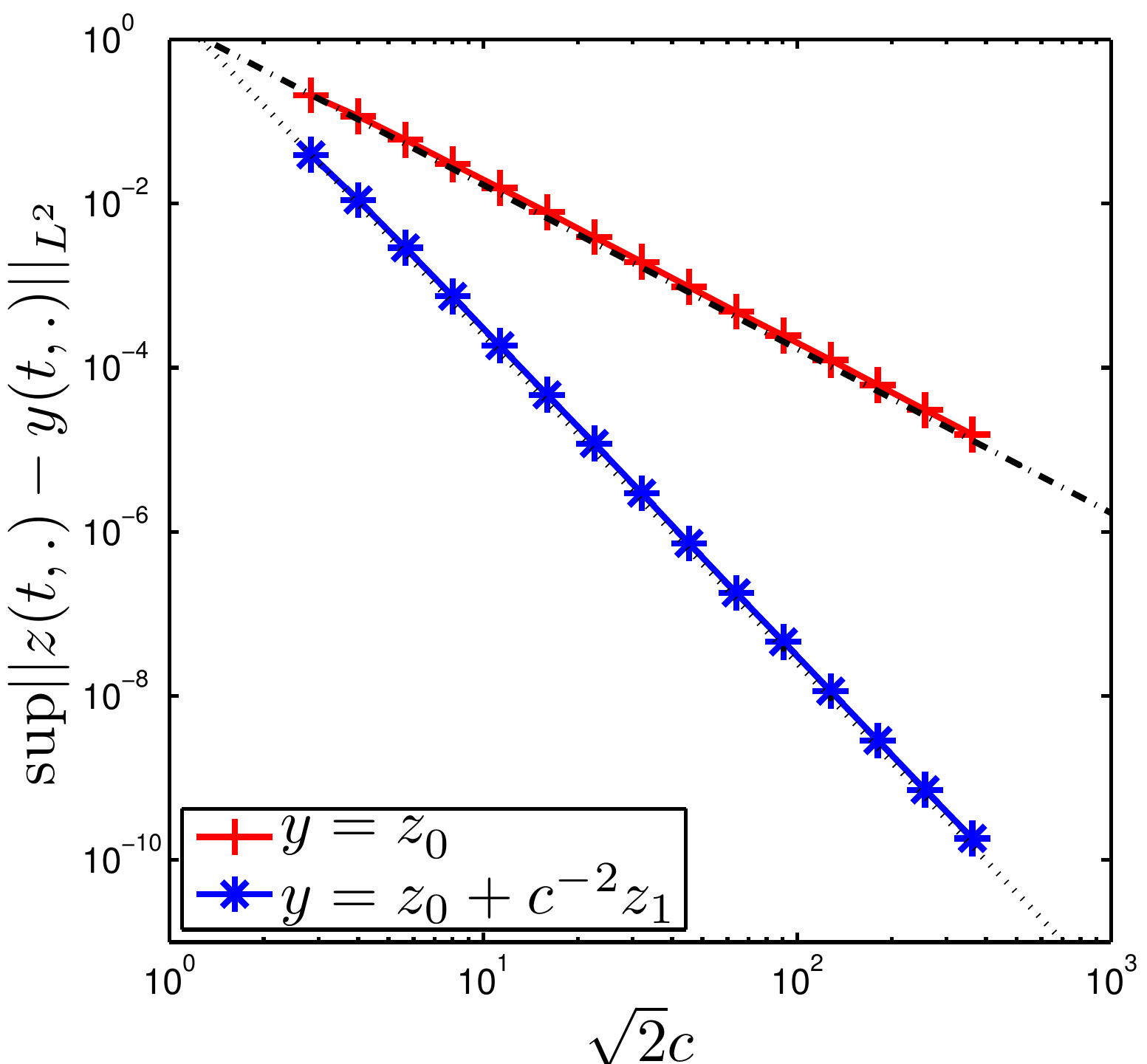}
\caption{Numerical simulations for the linear Klein-Gordon equation: Difference between the exact solution $z$ and the first- and second-order approximation terms $y = z_0$ (red, cross) and $y = z_0 + c^{-2}z_1$ (blue, star). The slope of the dash-dotted and dotted line is two and four, respectively.}\label{figure0}
\end{figure}

\subsection{Cubic non-linearity}\label{sec:cub}

In the cubic setting, i.e. $p = 1$, our model problem reads
\begin{equation}
\label{eq:kgrc}
\frac{1}{c^2} \partial_{tt} z - \Delta z + c^2 z = \lambda \vert z\vert^2 z, \quad z(0) = \varphi, \quad  \partial_t z(0) =c^2\gamma
\end{equation}
and the limit Schr\"odinger system \eqref{eq:nsl0} reduces to
\begin{equation}\label{eq:nsl0c}
\begin{aligned}
&i \partial_t u_0 - \frac{1}{2}\Delta u_0 - \frac{\lambda}{8}\big( \vert u_0\vert^2+2\vert v_0\vert^2\big)u_0 =0,\qquad u_0(0) = \varphi - i \gamma, \\
&i \partial_t v_0 - \frac{1}{2}\Delta v_0 -\frac{\lambda}{8}\big( \vert v_0\vert^2+2\vert u_0\vert^2\big)v_0=0,\qquad v_0(0) = \overline{\varphi + i \gamma},
\end{aligned}
\end{equation}
i.e. the first-order correction term $z_0$ is given by
\begin{equation}\label{eq:corr1}
z_0(t) = \textstyle \frac{1}{2} u_0(t)\mathrm{e}^{ic^2t} + \frac12 \overline{v_0}(t)\mathrm{e}^{-ic^2t},
\end{equation}
where $(u_0, v_0)$ solves the system of nonlinear Schr\"odinger equations \eqref{eq:nsl0c}. Here the enhancement of our Ansatz is clearly seen: the limit Schr\"odinger system can easily be simulated without any $c$-dependent time step restriction, and a well suited approximation to the exact solution in the non-relativistic limit regime is then obtained by simply multiplying the numerical approximations of $u_0$ and $v_0$ with the highly-oscillatory phases $\mathrm{e}^{ic^2t}$ and $\mathrm{e}^{-ic^2t}$, respectively.

For the numerical time integration of the first-order limit system, i.e. the Schr\"odinger system \eqref{eq:nsl0c}, we apply an exponential Strang splitting method, where we split the limit system in a natural way into the kinetic
\begin{equation}\label{eq:NSLT} 
i\partial_t \begin{bmatrix} u_0 \\ v_0\end{bmatrix} = \frac{1}{2}\Delta\begin{bmatrix} u_0 \\ v_0\end{bmatrix} 
\end{equation}
and potential part
\begin{equation}\label{eq:NSLP} 
i\partial_t \begin{bmatrix} u_0 \\ v_0\end{bmatrix} = \frac{\lambda}{8}\begin{bmatrix} \vert u_0\vert^2+ 2\vert v_0\vert^2 & 0\\ 0 & \vert v_0\vert^2 +2 \vert u_0\vert^2 \end{bmatrix}\begin{bmatrix} u_0 \\ v_0\end{bmatrix}.
\end{equation}
The exact flow $\Phi^t(u_0(0),v_0(0))= \Phi^t_{T+P}(u_0(0),v_0(0))$ of the coupled system \eqref{eq:nsl0c} at time $t_n = n\tau$ is then approximated by
\begin{equation}\label{eq:Lie}
\Phi^{t_n}\approx (\Phi^{\tau/2}_P\circ\Phi^{\tau}_T \circ \Phi^{\tau/2}_P)^n,
\end{equation}
where $\Phi_T^t(u_0(0),v_0(0))$ and $\Phi_P^t(u_0(0),v_0(0))$ are the exact flows associated to the kinetic Hamiltonian \eqref{eq:NSLT} and potential Hamiltonian \eqref{eq:NSLP}, respectively.

Note that the kinetic equation \eqref{eq:NSLT} can be solved exactly in Fourier space. Furthermore note that in the potential equation \eqref{eq:NSLP} the moduli of $u_0$ and $v_0$ are conserved quantities, i.e. $\vert u_0(t)\vert^2 = \vert u_0(0)\vert^2$ an $\vert v_0(t)\vert^2 = \vert v_0(0)\vert^2$ for all $t$ and thus we can solve both split equations in an exact way. 

For the space discretization, we use pseudo-spectral Fourier with mesh-size $h$ and grid points $x_j = {2jh\pi}$, $j = -K,\ldots,K-1$ where $K = 1/h \in \N$ is the number of frequencies. Note the fully discrete scheme can be easily implement using the Fast Fourier transform algorithm. We refer to \cite{Lubich08,Faou12} for convergence results over finite time of such splitting schemes.

\begin{example}[First-order approximation term]\label{ex:firstCorr}
In the first example we are interested in the asymptotic expansion up to the first-order correction term $z_0$. The asymptotic expansion derived in the previous section allows us the following convergence result for the natural choice of numerical approximation to $z_0$.
\begin{theorem}\label{thm:con0}
Let $s > d/2$ and assume that $\varphi$ and $\gamma$ are smooth functions defined on the one-dimensional torus $\T$. 
Then there exist constants $T, c_0, C, h_0$ and $\tau_0$ such that the following hold: for $c > c_0$,  $\tau < \tau_0$ and $h < h_0$, if we define by
\begin{equation}\label{eq:z0num}
z^{n,h}_0 := \textstyle \frac12 u_0^{n,h} \mathrm{e}^{ic^2t_n}+\frac12\overline{v_0^{n,h}}\mathrm{e}^{-ic^2t_n}
\end{equation}
the numerical approximation of the first-order approximation term $z_0(t)$ at time $t_n = n\tau$, where $u_0^{n,h}$ and $v_0^{n,h}$ denote the numerical approximations to the first-order limit system \eqref{eq:nsl0c} obtained by the Strang splitting \eqref{eq:Lie} at time $t_n = n\tau$ with a discrete grid of mesh-size $h$; then if $z(t)$ denotes the exact solution of the equation \eqref{eq:kgrc}, we have 
\begin{equation*}
\Vert z(t_n) -z^{n,h}_0\Vert_{L^2} \leq C( \tau^2 + h^{s} + c^{-2})\text{ for all } t_n \leq T.
\end{equation*}
\end{theorem}
\begin{proof}
We have
\begin{equation}\label{eq:boundi}
\Vert z(t_n) - z^{n,h}_0\Vert_{L^2} \leq \Vert z(t_n) - z_0(t_n)\Vert_{L^2} + \Vert z_0(t_n) - z^{n,h}_0\Vert_{L^2}.
\end{equation}
By Theorem \ref{thm:main} we have over the time interval $(0,T)$ and $c > c_0$,
\begin{equation}\label{eq:boundi1}
\Vert z(t_n) - z_0(t_n)\Vert_{L^2}\leq C  c^{-2}.
\end{equation}
To bound the second term we use the following inequality
\begin{equation*}
\begin{aligned}
\Vert z_0(t_n) - z^{n,h}_0\Vert_{L^2} &\leq  \Vert\mathrm{e}^{ic^2t_n}( u_0(t_n) - u_0^{n,h})\Vert_{L^2}+\Vert \mathrm{e}^{ic^2t_n}(v_0(t_n) - v_0^{n,h})\Vert_{L^2}\\
&\leq  \Vert  u_0(t_n) - u_0^{n,h} \Vert_{L^2}+\Vert v_0(t_n) - v_0^{n,h} \Vert_{L^2}.
\end{aligned}
\end{equation*}
Using arguments similar to  \cite[Theorem IV.17]{Faou12} (see also \cite{Lubich08}), it is then possible to prove that 
\begin{equation}
\label{eq:umun}
\Vert   u_0(t_n) - u_0^{n,h}\Vert_{L^2}+\Vert  v_0(t_n) - v_0^{n,h}\Vert_{L^2} \leq C( \tau^2 + h^{s} ), 
\end{equation}
if the exact solution is smooth enough (i.e. belongs to some Sobolev space $H^{s'}$ with $s' > s$ sufficiently large). Note that the proof of the fully discrete scheme in \cite{Faou12} is given only for the Lie-splitting, real initial data (that is $\varphi$ and $\gamma$ real valued functions) and uses function spaces based on the Wiener algebra (the $\ell_s^1$ spaces). However, the rigorous proof of \eqref{eq:umun} is essentially a variation on the same theme and does not present any particular difficulty as soon as the solution is smooth enough. 
Using this bound, we obtain that 
\begin{equation}\label{eq:boundi2}
\Vert z_0(t_n) - z^{n,h}_0\Vert_{L^2} \leq C( \tau^2 + h^{s} ).
\end{equation}
Plugging the bounds \eqref{eq:boundi1} and \eqref{eq:boundi2} into \eqref{eq:boundi} yields the desired convergence result.
\end{proof}

The numerical results for the initial data 
$$ \varphi = \frac{(2+i)}{\sqrt{5}}\cos(x),\,\gamma = \frac{(1+i)}{\sqrt{2}}\sin(x)+\frac{1}{2}\cos(x)
$$ 
and $\lambda = -1$ on the time interval $[0,0.1]$ are given in Figure \ref{figure1}.
As a reference solution for $z(t)$ we take the numerical approximation of the Klein-Gordon equation \eqref{eq:kgrc} obtained with an exponential Gautschi method, proposed in \cite{BaoD12}, with a very small time step $\tau$ satisfying the necessary condition $\tau < c^2 h$. Note that the convergence result for the exponential Gautschi method given in \cite[Theorem 9]{BaoD12}
shows a convergence rate of order $\mathcal{O} (\tau^2 c^4 + h^{s})$ under some {\em a priori} smoothness assumptions and uniform bounds for the exact solution. 

\begin{figure}[h!]
\centering
\includegraphics[width=0.52\linewidth]{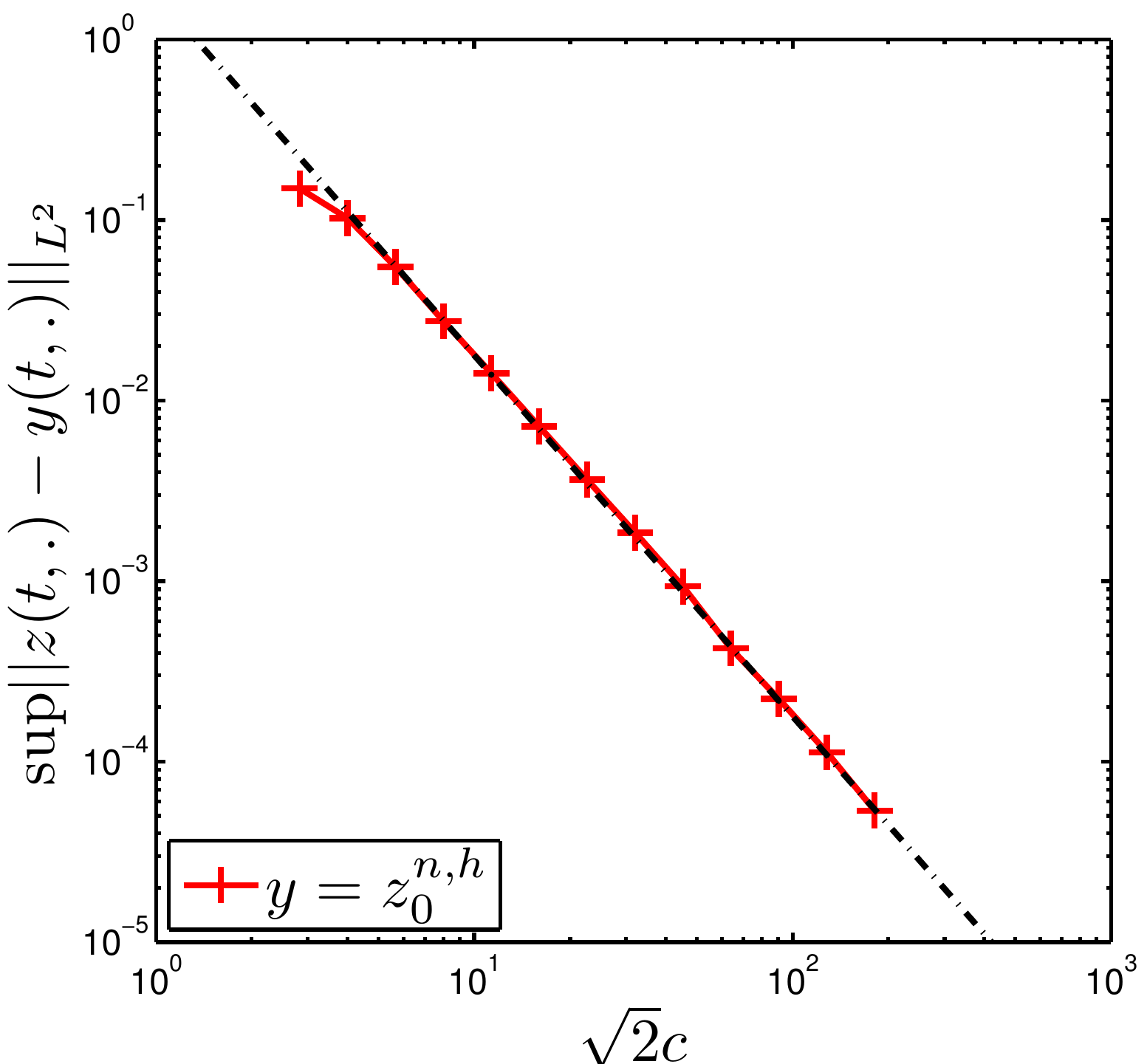}
\caption{Numerical simulations of Example \ref{ex:firstCorr}: Difference between the reference solution $z$ obtained with an exponential Gautschi method with time step $\tau = 10^{-6}$ and the first-order approximation solution $y= z_0^{n,h}$ obtained with $\tau = 10^{-2}$. The slope of the dash-dotted line is two.}\label{figure1}
\end{figure}
\end{example}

\begin{example}[Second-order approximation term]\label{ex:2ndCorr}
In this example we simulate the first- and second-order approximation terms $z_0$ and $z_1$ given in \eqref{eq:corr1} and \eqref{eq:corr2}, respectively. For the numerical approximation of the first approximation term we proceed as in Example \ref{ex:firstCorr}. For the numerical approximation of the correction term $z_1$, we solve the equation for $\xi_1$, i.e. equation \eqref{eq:u1rel}, with a symmetric Strang splitting based on the following split-equations: The kinetic equation
\begin{equation*}
i\partial_t \xi =  \textstyle \frac12 \Delta \xi,
\end{equation*}
which can be solved exactly in Fourier, and the linear potential equation
\begin{equation}
\label{eq:poteq}
i\partial_t \xi = \lambda \textstyle \frac34 \vert u_0\vert^2 \xi +\lambda \frac38 (u_0)^2\overline{\xi}+ \textstyle \frac18 \Delta^2 u_0 +\lambda^2 \frac{51}{256}\vert u_0\vert^4 u_0 +\lambda \frac{3}{16}\Delta(\vert u_0\vert u_0),
\end{equation}
which we solve as described in the following. We set $\alpha_0 = \mathrm{Re}(u_0)$, $\beta_0 =\mathrm{Im}(u_0)$, $\alpha =  \mathrm{Re}(\xi)$, $\beta = \mathrm{Im}(\xi)$,  and furthermore $g_0(t) = g(\alpha_0(t),\beta_0(t)) = \textstyle \frac18 \Delta^2 u_0 +\lambda^2 \frac{51}{256}\vert u_0\vert^4 u_0 +\lambda \frac{3}{16}\Delta(\vert u_0\vert u_0)$. Then we can rewrite the potential equation as
\begin{equation}\label{eq:poti}
\begin{aligned}
\partial_t \begin{bmatrix} \alpha \\ \beta\end{bmatrix} = -\lambda \frac38\begin{bmatrix} 2 \alpha_0\beta_0\qquad &( (\alpha_0)^2+3 (\beta_0)^2)\\ - ((\beta_0)^2+3 (\alpha_0)^2)\qquad &- 2 \alpha_0\beta_0\end{bmatrix}\begin{bmatrix} \alpha \\ \beta\end{bmatrix} +\begin{bmatrix}\mathrm{Im}(g_0)\\ -\mathrm{Re}(g_0)\end{bmatrix}.
\end{aligned}
\end{equation}
We set
\begin{equation*}
A(\alpha_0,\beta_0) = -\lambda \frac38\begin{bmatrix} 2 \alpha_0\beta_0\qquad &( (\alpha_0)^2+3 (\beta_0)^2)\\ - ((\beta_0)^2+3 (\alpha_0)^2)\qquad &- 2 \alpha_0\beta_0\end{bmatrix}
\end{equation*}
and approximate the potential equation \eqref{eq:poti} with an exponential trapezoidal rule, i.e. for $\alpha^n_k \approx \alpha_k(t_n)$, $\beta^n_k \approx \beta_k(t_n)$, $k=0,1$, $g^n = g(\alpha^n_0,\beta^n_0)$, $t_n = n\tau$, we set
\begin{multline}
\label{eq:num2corr}
\begin{bmatrix} \alpha^{n+1}\\\beta^{n+1}\end{bmatrix}= \mathrm{e}^{\frac{\tau}{2}\big(A(\alpha_0^{n+1},\beta_0^{n+1})+A(\alpha_0^{n},\beta_0^{n})\big)}\textstyle\left(\begin{bmatrix} \alpha^{n}\\ \beta^{n}\end{bmatrix} + \frac{\tau}{2}\begin{bmatrix}\mathrm{Im}(g_0^n)\\ -\mathrm{Re}(g_0^n)\end{bmatrix} \right) \\� +\textstyle\frac{\tau}{2}\begin{bmatrix}\mathrm{Im}(g_0^{n+1})\\ -\mathrm{Re}(g_0^{n+1})\end{bmatrix}.
\end{multline}
The numerical approximation $\xi^{n+1}$ to the exact solution $\xi(t_{n+1})$ of the potential equation \eqref{eq:poteq} starting in a given point $\xi^n$ is then simply given by $\xi^{n+1} = \alpha^{n+1} + i \beta^{n+1}$. 

Concerning the space discretization, we use again the Fourier pseudo-spectral method described in the previous Section, and we denote by $h$ the corresponding mesh size. 
Note that in this setting, the approximation of $g_0$ can be easily made using FFT.  

Let $u_0^{n,h}$ be defined as in the previous section, and let $\xi^{n,h}_1$ denote the numerical approximation to the full system \eqref{eq:u1rel} at time $t_n$ obtained by the fully discrete Strang splitting described above, where we make half a step with the kinetic, a full step with the potential and again half a step with the kinetic part. Then, as the numerical approximation to the exact second-order approximation term  at time $t_n = n\tau$ we take
\begin{multline}
\label{eq:approxnh}
z_0^{n,h} + c^{-2}z_1^{n,h}:=  \textstyle \frac12(1+\frac{1}{c^2}\frac{3\lambda}{16}\vert u^{n,h}_0\vert^2) (u_0^n\mathrm{e}^{ic^2t_n}+\overline{u_0^{n,h}}\mathrm{e}^{-ic^2t_n}) \\�
- \textstyle\frac{\lambda}{64 c^2}\big( (u_0^{n,h})^3\mathrm{e}^{3ic^2t_n} +(\overline{u_0^{n,h}})^3\mathrm{e}^{-3ic^2t_n}\big)
 +  \textstyle \frac{1}{2c^2} \big( \xi_1^{n,h} \mathrm{e}^{ic^2t_n}+\overline{\xi_1^{n,h}}\mathrm{e}^{-ic^2t_n}\big).
\end{multline}
For this approximation we have the following convergence result.
\begin{theorem}
Let $s > d/2$ and assume that $\varphi$ and $\gamma$ are smooth functions defined on the one-dimensional torus $\T$. 
Then there exist constants $T, c_0, C, h_0$ and $\tau_0$ such that the following hold: for $c > c_0$,  $\tau < \tau_0$ and $h < h_0$, then if $z(t)$ denotes the exact solution of the equation \eqref{eq:kgrc}, we have 
\begin{equation*}
\Vert z(t_n) -(z_0^{n,h} + c^{-2}z_1^{n,h})\Vert_{L^2} \leq C( \tau^2 + h^{s} + c^{-4})\text{ for all } t_n \leq T,
\end{equation*}
where the approximation $z_0^{n,h} + c^{-2}z_1^{n,h}$ is defined in \eqref{eq:approxnh}. 
\end{theorem}
\begin{proof}
Note that by Theorem \ref{thm:main} we have over the time interval $(0,T)$ and $c>c_0$,
$$ \Vert z(t_n) - (z_0(t_n) + c^{-2}z_1(t_n))\Vert_{L^2} \leq C c^{-4}.$$
The rest of the proof follows the line of argumentation given for the proof to Theorem \ref{thm:con0}, since the scheme described in \eqref{eq:num2corr} is second-order convergent. We do not give the detail of this proof. Again it is a variation on the same theme as in \cite{Faou12,Lubich08}. 
\end{proof}
The numerical results for the initial data 
$$
\varphi = \cos(x), \,\gamma =\frac14 \sin(x)+\frac12 \cos(x)
$$
and $\lambda = -1$ on the time interval $[0,0.1]$ are given in Figure \ref{figure2b}.
\begin{figure}[h!]
\centering
\includegraphics[width=0.52\linewidth]{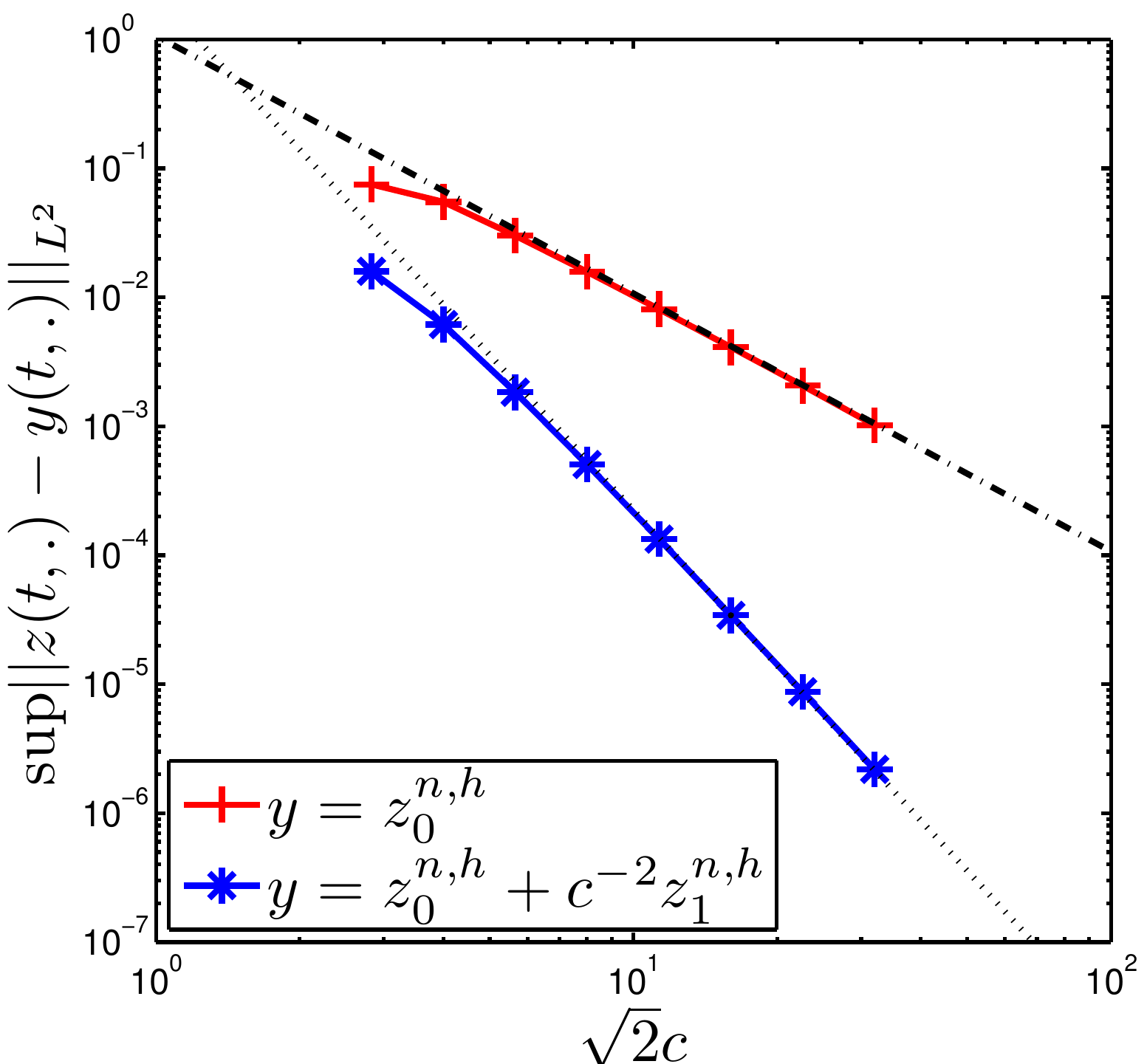}
\caption{Numerical simulations of Example \ref{ex:2ndCorr}: Difference between the reference solution $z$ obtained with an exponential Gautschi method with time step $\tau = 10^{-7}$ and the first- and second-order approximation solutions $y = z_0^{n,h}$ (red , cross) and $y = z_0^{n,h} + c^{-2}z_1^{n,h}$ (blue, star) obtained with $\tau = 10^{-3}$. The slope of the dash-dotted and dotted line is two and four, respectively.}\label{figure2b}
\end{figure}
\end{example}

\begin{example}[Energy and charge-density deviation]\label{ex:EQ}
Here we simulate the deviation of the the first energy and charge correction term \eqref{eq:E0} and \eqref{eq:Q0}, respectively. Also, we simulate the energy and charge deviation of the first-correction term $z_0(t)$, i.e. $E(z_0(t))$ and $Q(z_0(t))$, where the energy and charge are defined in \eqref{eq:eni} and \eqref{eq:chargi}, respectively.

We set $E(z_0^{n,h})= E(u_0^{n,h}\mathrm{e}^{ic^2t_n},v_0^{n,h}\mathrm{e}^{ic^2t_n})$ the numerical energy of the first correction term $z_0$,  for $t_n = n\tau$, and where $u_0^{n,h}$ and $v_0^{n,h}$ are the numerical approximations of the first limit system \eqref{eq:nsl0c} again obtained by the Strang splitting \eqref{eq:Lie}. Similarly, we set $\Ec_0^{n,h}= \Ec_0(u_0^{n,h}\mathrm{e}^{ic^2t_n},v_0^{n,h}\mathrm{e}^{ic^2t_n})$ the approximation of the first term in the asymptotic expansion of the energy. In the same way, we define the numerical charge $Q(z_0^{n,h})= Q(u_0^{n,h}\mathrm{e}^{ic^2t_n},v_0^{n,h}\mathrm{e}^{ic^2t_n})$ of the first correction term, as well as the numerical approximation of the first term in the asymptotic expansion 
 $Q_0^{n,h} = Q_0(u_0^{n,h},v_0^{n,h})$.
 
We carry out the simulations on the time interval $[0,1]$ with the time step $\tau = 10^{-3}$. As the initial values we choose the same as in Example \ref{ex:firstCorr}, but normalize them with respect to the $H^1$ norm. We nicely see that the first energy correction term $\Ec_0$ and the first charge correction term $Q_0$ are conserved by the Strang splitting, see Figure \ref{figureE} and Figure \ref{figureQ}, respectively. Furthermore, Figure \ref{figureQz0} and Figure \ref{figureEz0} underline that $\vert Q(z_0^{0,h})-Q(z_0^{n,h})\vert = \mathcal{O}(c^{-2})$ and $\vert E(z_0^{0,h})-E(z_0^{n,h})\vert = \mathcal{O}(1)$.
\begin{figure}[h!]
\centering
\includegraphics[width=0.99\linewidth]{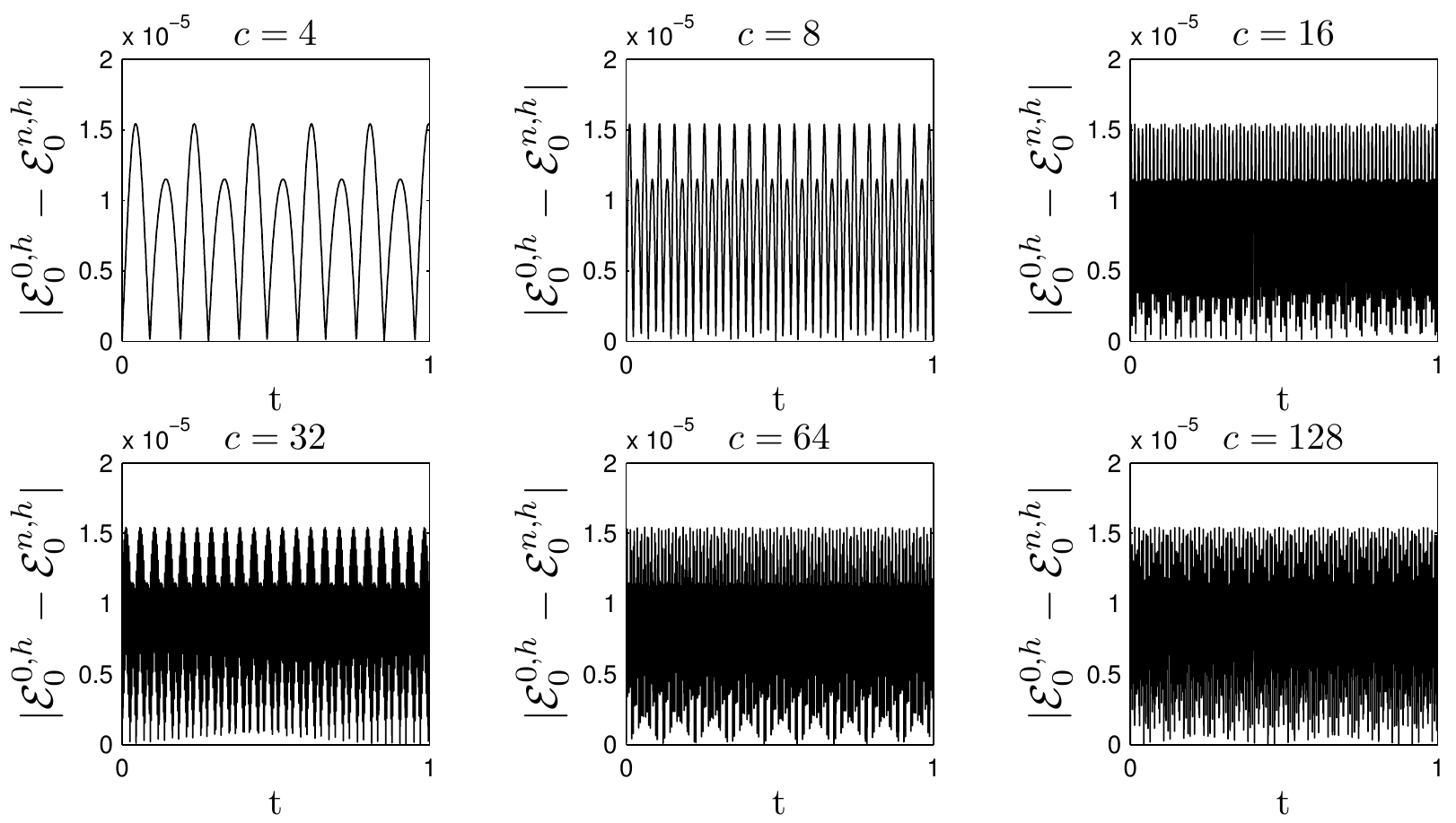}
\caption{Numerical Simulation of Example \ref{ex:EQ}: Energy deviation of the first correction term $\Ec_0^{n,h}$, for different values of oscillations $c$.}\label{figureE}
\end{figure}
\begin{figure}[h!]
\centering
\includegraphics[width=0.52\linewidth]{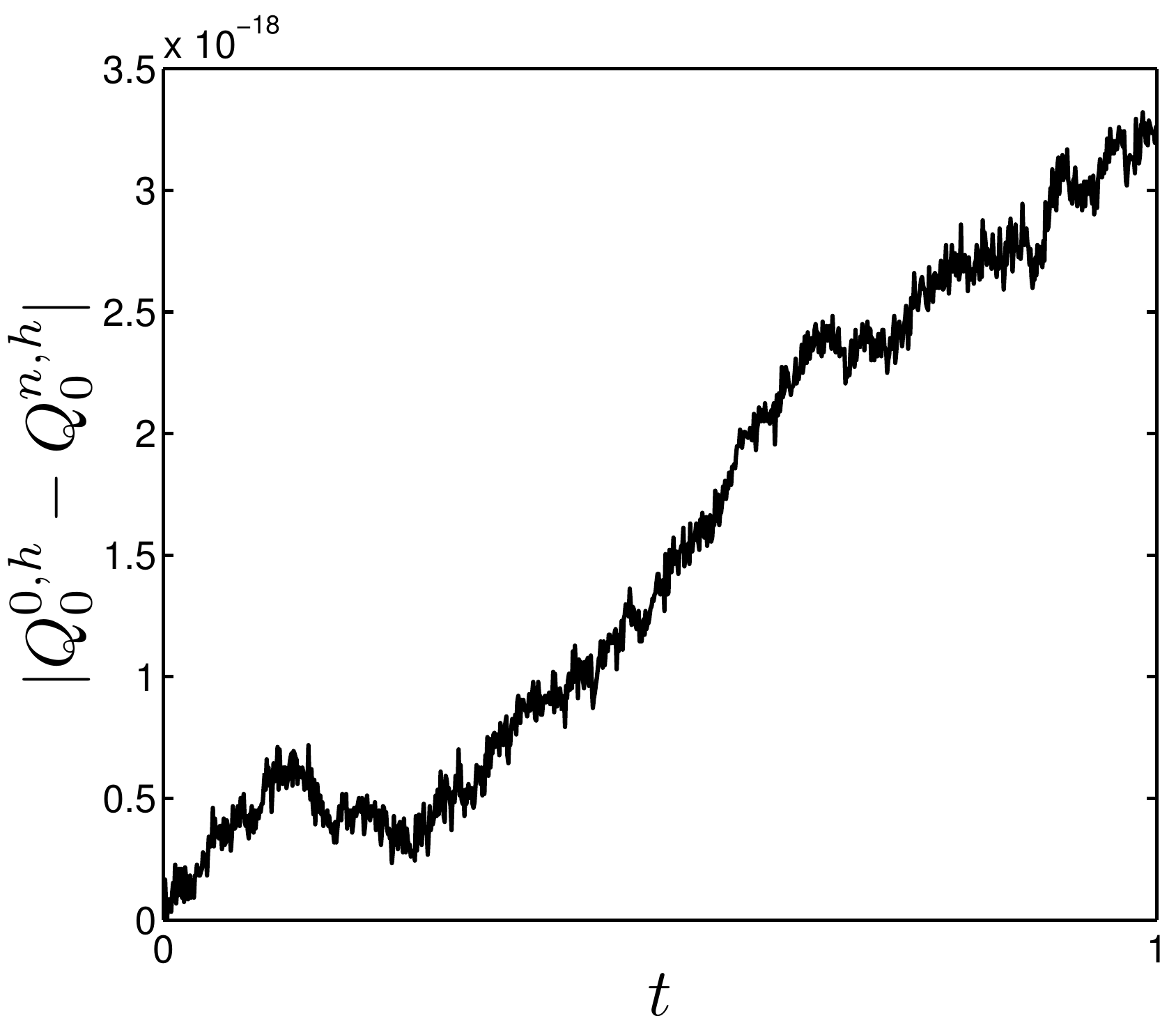}
\caption{Numerical Simulation of Example \ref{ex:EQ}: Charge deviation of the first correction term $Q_0^{n,h}$.}\label{figureQ}
\end{figure}
\end{example}

\begin{figure}[h!]
\centering
\includegraphics[width=0.99\linewidth]{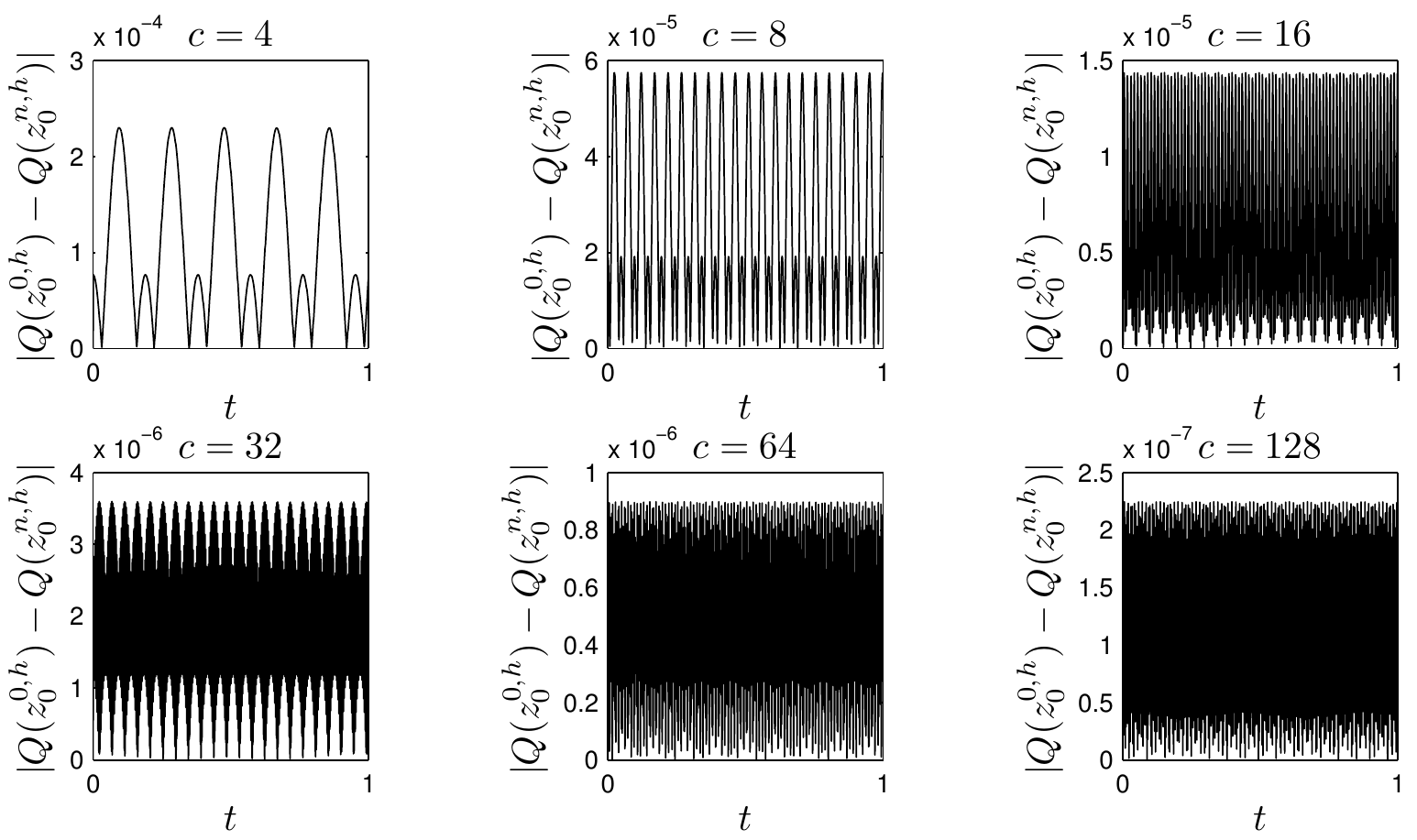}
\caption{Numerical Simulation of Example \ref{ex:EQ}: Charge deviation of $Q(z_0^{n,h})$.}\label{figureQz0}
\end{figure}

\begin{figure}[h!]
\centering
\includegraphics[width=0.99\linewidth]{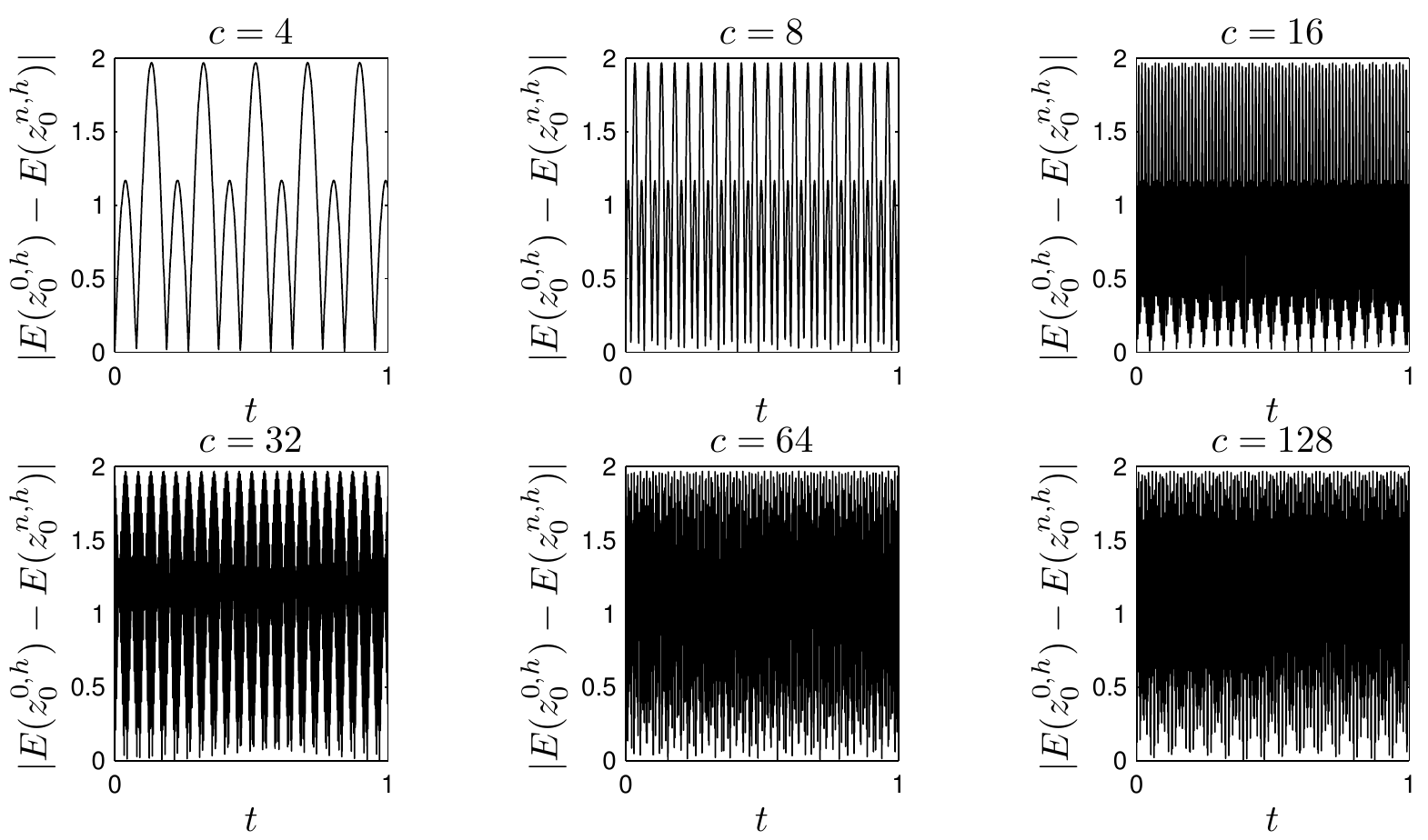}
\caption{Numerical Simulation of Example \ref{ex:EQ}: Energy deviation of $E(z_0^{n,h})$.}\label{figureEz0}
\end{figure}

\end{document}